\documentclass[a4paper,11pt]{article}
\usepackage{dsfont}
\usepackage{amssymb}
\usepackage{latexsym}
\usepackage{amsmath}
\DeclareMathOperator*{\esssup}{ess\,sup}
\usepackage{comment}
\usepackage{amsthm}
\usepackage[dvips]{graphicx}
\usepackage{geometry}
\usepackage{layout} 
\usepackage{enumitem}
\usepackage{bm}
\usepackage{bbm} 
\usepackage[bbgreekl]{mathbbol} 
\usepackage{authblk}
\usepackage{setspace} 
\usepackage[dvipsnames]{xcolor}
\usepackage[colorlinks=true, linkcolor=blue, urlcolor=blue]{hyperref}

\usepackage[style=ieee, natbib=true]{biblatex} 
\newif\ifrs
\rstrue
\ifrs \usepackage{mathrsfs} \fi  
\newif\ifcol
\coltrue 

\newtheorem{theorem}{Theorem}[section]

\newtheorem{lemma}[theorem]{Lemma}

\newtheorem{remark}[theorem]{Remark}
\newtheorem{example}[theorem]{Example}
\numberwithin{equation}{section}
\newtheorem{theorem*}{Theorem}
\newtheorem{ass*}[theorem*]{Assumption}
\newtheorem{note*}[theorem*]{Note}
\newtheorem{lemma*}[theorem*]{Lemma}
\newtheorem{definition*}[theorem*]{Definition}
\newtheorem{proposition*}[theorem*]{Proposition}
\newtheorem{corollary*}[theorem*]{Corollary}
\newtheorem{remark*}[theorem*]{Remark}
\newtheorem{example*}[theorem*]{Example}
\numberwithin{equation}{section}
\newcommand{\bbx}{{\mathbb x}}

\newcommand{\bbz}{{\mathbb z}}
\newcommand{\checkN}{N}
\newcommand{\consta}{C_1}%

\newcommand{\constc}{C_0}%
\newcommand{\constd}{{c_1}}%

\newcommand{\consth}{{XC_4}}
\newcommand{\consti}{{C_1'}}%
\newcommand{\constj}{{C_1'''}}%
\newcommand{\constk}{{c_2}}
\newcommand{\constl}{{C_6}}
\newcommand{\constm}{{C_7}}
\newcommand{\constn}{{C_{2,r}}}
\newcommand{\consto}{{C_{3,r}}}
\newcommand{\constp}{{C_4}}
\newcommand{\constq}{{C_{5}}}%
\newcommand{\constr}{{c_{3}}}
\newcommand{\constt}{{C_1''}}%
\newcommand{\constu}{{C_{9}}}
\newcommand{\constv}{{c_1 + \beta}}%
\newcommand{\constz}{{C_{10}}}
\newcommand{\lG}{{l_{\mathrm{g}}}}
\newcommand{\lS}{{l_{\mathrm{s}}}}
\newcommand{\lN}{{l_{\mathrm{n}}}}
\newcommand{\lA}{{l_{\mathrm{A}}}}

\coltrue
\ifcol
\newcommand{\colorr}{\color[rgb]{0.8,0,0}}

\newcommand{\colorn}{\color[rgb]{1,1,1}}

\newcommand{\colory}{\color{yellow}}

\else

\newcommand{\colorr}{\color{black}}

\newcommand{\colorn}{\color{black}}
\newcommand{\colory}{\color{black}}

\fi
\newif\ifcol
\colfalse
\ifcol

\else
\fi
\newif\ifcol
\colfalse
\ifcol

\else
\fi
\newif\ifcol
\colfalse
\ifcol

\else
\fi
\newif\ifcol
\coltrue
\ifcol

\else
\fi
\excludecomment{en-text}
\includecomment{jp-text}
\includecomment{comment}
\excludecomment{comment}
\def\koko{{\coloroy{koko}}}
\def\bd{\begin{description}}
\def\ed{\end{description}}

\def\D2{\bbD_{2,\infty-}}

\def\D{{\bf D}}

\def\calb{{\cal B}}
\def\calc{{\cal C}}

\def\calf{{\cal F}}

\def\calh{{\cal H}}
\def\cali{{\cal I}}

\def\caln{{\cal N}}

\def\ds{\displaystyle}
\def\yeq{\>=\>}
\def\yleq{\>\leq\>}
\def\ygeq{\>\geq\>}

\def\sfk{{\sf k}}

\def\sfd{{\sf d}}
\def\sfp{{\sf p}}

\def\ep{\epsilon}
\def\half{\frac{1}{2}}

\def\down{\downarrow}

\def\nn{\nonumber}
\def\be{\begin{equation}}
\def\ee{\end{equation}}
\def\bea{\begin{eqnarray}}
\def\eea{\end{eqnarray}}
\def\beas{\begin{eqnarray*}}
\def\eeas{\end{eqnarray*}}
\def\bi{\begin{itemize}}
\def\ei{\end{itemize}}
\def\im{\item}
\def\bd{\begin{description}}
\def\ed{\end{description}}
\newcommand{\bbD}{{\mathbb D}}

\newcommand{\bbF}{{\mathbb F}}

\newcommand{\bbN}{{\mathbb N}}

\newcommand{\bbR}{{\mathbb R}}
\newcommand{\bbS}{{\mathbb S}}

\newcommand{\bbX}{{\mathbb X}}

\newcommand{\thL}{\hat \theta_\epsilon^L}
\newcommand{\thLi}{\hat \theta_{\epsilon, ij}^L}

\renewcommand{\colory}{\color{yellow}}
\renewcommand{\koko}{{\colory koko}}
\newcommand{\wh}{\widehat}
\newcommand{\wt}{\widetilde}

\newcommand{\gray}{\color[rgb]{0.5,0.5,0.5}}

\makeatletter
\newcommand{\setitem}[1]{\item[#1]\def\@currentlabel{#1}}
\makeatother

\begin{document}

\title{
Network inference for SDEs with diverging dimension under small noise
\date{}

\let\thefootnote\relax\footnotetext{
This work was in part supported by 
Japan Science and Technology Agency CREST JPMJCR2115;  
Japan Society for the Promotion of Science Grants-in-Aid for Scientific Research 
No. 23H03354 (Scientific Research);  
Forefront Physics and Mathematics Program to Drive Transformation (FoPM), a World-leading Innovative Graduate Study (WINGS) Program, the University of Tokyo; 
and by a Cooperative Research Program of the Institute of Statistical Mathematics. 
}
}
\author[2]{Francesco Iafrate}
\author[1]{Nakahiro Yoshida}
\author[3]{Stefano M. Iacus}
\affil[3]{ \small European Centre for Algorithmic Transparency
\footnote{ Joint Research Centre of the European Commission, Seville, Spain}}

\affil[2]{ \small Department of Mathematics, University of Hamburg
\footnote{University of Hamburg, Bundesstr. 55, 20146, Hamburg, Germany. e-mail: francesco.iafrate@uni-hamburg.de}}

\affil[1]{\small Graduate School of Mathematical Sciences, University of Tokyo
\footnote{Graduate School of Mathematical Sciences, University of Tokyo: 3-8-1 Komaba, Meguro-ku, Tokyo 153-8914, Japan. e-mail: nakahiro@ms.u-tokyo.ac.jp}}
\affil[1]{CREST, Japan Science and Technology Agency
}
\maketitle
\ \\
\begin{abstract}
We consider inference for network stochastic differential equations in a small-noise, high-dimensional regime, where the diffusion coefficient vanishes with $\ep\to0$ while the network size and parameter dimension may diverge. We develop a minimum-distance estimation framework based on a deterministic measurement model that may omit nuisance terms in the drift, and derive a non-asymptotic error bound decomposing the estimation error into stochastic fluctuation, nuisance discrepancy, and identifiability terms.

For a linearly parametrized interaction model under repeated measurements, we obtain explicit non-asymptotic bounds on the nuisance discrepancy and the identifiability gap in terms of graph connectivity and noise level. Under suitable mixing and balance conditions, this yields consistency and convergence rates for the estimator. We further study graph recovery via adaptive Lasso. The penalized estimator is shown to inherit the rate of the preliminary estimator and to recover the true interaction graph consistently. These results provide a theoretical basis for parameter estimation and edge selection in high-dimensional network SDEs with nuisance drift components.
\end{abstract}


\section{Introduction}


Many dynamical systems of interacting components are naturally represented through a network structure, where each component evolves according to its own state and the states of neighboring components. Examples arise in biology, neuroscience, finance, and interacting particle systems, where the main inferential goal is often not only the estimation of drift parameters, but also the recovery of the underlying interaction pattern, especially for large-scale systems (\cite{stapor2022mini}). 

In recent years, network interaction structures in systems of stochastic differential equations have been studied. 
Network stochastic differential equations (N-SDEs) provide a continuous-time framework for modeling systems of interacting random processes, where the interaction pattern is represented by a graph. Let $G=(V,E)$ be a graph with node set $V=[\sfd]$, let $X_t^i\in\mathbb R$ denote the state of node $i$ at time $t$ and let $N_i$ be the neighbors of node $i$. 
Such models have been introduced in the context of autoregressive models in 
\cite{nar}, and extended to the nonlinear case in \cite{armillotta2023nonlinear}. Network dependence in SDEs has been studied in 
\cite{ou-graph1, ou-graph2}. In \cite{iafrate2024ergodic} the authors introduced a general  nonlinear N-SDE of the form
\[
dX_t^i
=
\left(
\underbrace{b_{ii}(X_t^i;\beta)}_{\text{momentum effect}}
+
\sum_{j\in N_i}
\underbrace{b_{ij}(X_t^i,X_t^j;\beta)}_{\text{network effect}}
\right)\,dt
+
\underbrace{\sigma_i(X_t^i,\alpha)}_{\text{node volatility}}\,dW_t^i, \qquad i = 1, \ldots, \sfd
\]
where $b_{ii}$ describes the intrinsic dynamics of node $i$, the functions $b_{ij}$ encode the interaction with neighboring nodes $j\in N_i$, $\sigma_i$ represent the node volatility, $W_1, \ldots W_d$ are independent Brownian motions, and the graph is possibly directed. The parameters $\theta=(\alpha, \beta)$ are unknown. This formulation is well suited to statistical questions concerning both parameter estimation and recovery of the underlying interaction graph.

In this paper, we study this problem in a small-noise regime, where the diffusion coefficient vanishes with a parameter $\ep\to0$, while the dimension of the system $\sfd_\ep$ is allowed to diverge and with it the neighborhood structure and the number of the interaction parameters $\theta_\ep$.  We also allow the drift to contain nuisance terms that are not directly estimable.
Our goal is to infer the interaction parameters and  the graph structure. Supposing that a continuous trajectory over $[0, T]$ has been observed, we adopt a minimum-distance approach based on a \textit{measurement model}: the observed stochastic system is compared with the solution of a deterministic ODE in which the nuisance part of the drift is omitted. 

Our first contribution is a general non-asymptotic error bound for minimum-distance estimators. Under weak conditions we show that the estimation error decomposes
into three components: a stochastic fluctuation term induced by the diffusion, a nuisance discrepancy term measuring the effect of the misspecified part of the drift, and an identifiability term measuring the separation of the measurement model away from the true parameter. 

The main difficulty in exploiting the general bound is the control of identifiability in high dimension.  To overcome this difficulty, we focus on a linearly parametrized principal interaction model and introduce a repeated-measurements scheme. 
Within this structured model, we derive explicit controls for the nuisance discrepancy and the identifiability term. This makes it possible to express the general error bound in terms of explicit quantities depending on the graph connectivity, the number of repeated measurements, the noise level, and a concentration event. Under suitable mixing assumptions and balance conditions linking graph growth, sample size, and noise magnitude, we obtain consistency of the minimum-distance estimator. 
In particular, we assume that the graph connectivity grows like $\mathcal O(\log \sfd_\epsilon)$. 
We then show that stronger scaling assumptions yield explicit convergence rates, including both a slow logarithmic rate and a faster polynomial rate. In particular, in the latter case we show for the estimator $\wh{\theta}_\ep$ a convergence of the form
\[
    \ep^{-\gamma}\,\rho_\ep(\wh{\theta}_\ep,\theta_\ep^*) \to 0
    \qquad\text{in probability as }\ep\to0
\]
where $\rho_\epsilon$ is a suitable metric on the parameter space. 
Our final contribution concerns graph recovery. After obtaining a preliminary estimator through the minimum-distance procedure, we apply an adaptive Lasso in order to induce sparsity and estimate the edge set. We show that the penalized estimator inherits the $\epsilon^{-\gamma}$ rate of the preliminary estimator under a suitable control of the penalty on active coordinates. Moreover, under additional separation conditions on the adaptive weights and a beta-min condition on the true nonzero coefficients, we prove support recovery consistency.

The present work is related to several lines of literature. On one side, it connects with inference for small-noise diffusions and trajectory-based estimation methods, where deterministic approximations of stochastic systems are used to derive consistent estimators, see  \cite{kutoyants2012identification}. See also  \cite{sorensen2003small} for a different approach based on a quasi-likelihood estimator in the setting of high-frequency observations. 
On the other side, it is related to high-dimensional sparse estimation and adaptive $\ell_1$-penalization,  \cite{zou2006adaptive}, extended to diverging dimension in \cite{huang2008adaptive}. 
Starting with \cite{de2012adaptive}, in recent years there has been a flourishing line of research papers dealing with regularization and high-dimensionality for diffusion processes 
\cite{suzuki2020penalized, de2021regularized,  amorino2025sampling, de2025adaptive, trottner2023concentration, marushkevych2025consistent,
nakakita2025sparse,
dexheimer2026sparse, bayraktar2026sparse}. In particular lasso-type estimators for small noise diffusions have been studied
in \cite{de2018penalized}.
The novelty here lies in combining these two perspectives in a network SDE framework with diverging dimension and nuisance misspecification.

The paper is organized as follows. Section \ref{sec:general-err} establishes a general error bound for minimum-distance estimators under a broad measurement-model framework. Section \ref{sec:ident} specializes the analysis to a linearly parametrized interaction model, introduces repeated measurements, and derives explicit identifiability and nuisance controls. The same section also proves consistency and convergence rates under suitable balance conditions. Section \ref{sec:lasso} studies adaptive-Lasso post-processing for graph recovery and proves both estimation consistency and support recovery. The proofs are collected in Section  \ref{sec:proofs}.

\section{Network SDE Model}

Let $(\Omega,\calf,\bbF,P)$ be a stochastic basis, with
$\bbF=(\calf_t)_{t\in[0,T]}$, and write $[k]=\{1,\ldots,k\}$.
For each $\ep\in(0,1]$, we consider network SDEs $(X^{\ep,i}_t)_{t\in[0,T], i\in[\sfd_\ep]}$ 
which are progressively measurable with respect to $\bbF$.
The interaction structure is encoded by a graph
$
G_\ep=(V_\ep,E_\ep),
\, V_\ep=[\sfd_\ep],
$ 
with adjacency matrix $\mathcal A_\ep$. Denote by $N_i^\epsilon$ the parents (incoming connections) of node $i$, possibly including node $i$ itself, corresponding to self-dynamics. 
For each node $i$, let
$\caln_i^\ep$ denote the collection of indices of the components entering the drift of
$X^{\ep,i}$, possibly including repeated indices. We write
\[
\bbX_t^{\ep,i}
=
(X_t^{\ep,k})_{k\in\caln_i^\ep},
\qquad
\sfk_i^\ep=\dim(\bbX_t^{\ep,i}).
\]
Thus $\sfk_i^\ep$ is the number of coordinates on which the drift of the $i$th component depends.
The collection $\caln_i^\ep$ contains the elements of $N_i^\epsilon$, possibly repeated
\footnote{Dealing with $\caln_i^\ep$ instead of the actual set of neighbors $N_i$ allows for a general simplified treatment, while maintaining the same intuition. In plain language we still refer to $\sfk_i^\ep$ as ``the number of neighbors" of node $i$.}.
See the examples below.

The parameter is denoted by
$\theta_\ep\in\Theta_\ep\subset\bbR^{\pi_\ep}$, 
where $(\Theta_\ep,\rho_\ep)$ is a compact metric space. The metric $\rho_\ep$ is not required to be Euclidean. For instance, one may take
\[
\rho_\ep(\theta^1,\theta^2)
=
\max_{\ell\in[\pi_\ep]}|\theta^1_\ell-\theta^2_\ell|,
\]
or, for $p\in[1,\infty)$ and a normalizing factor $a_\ep>0$,
\[
\rho_\ep(\theta^1,\theta^2)
=
a_\ep
\left\{
\sum_{\ell\in[\pi_\ep]}|\theta^1_\ell-\theta^2_\ell|^p
\right\}^{1/p} \qquad 
(\theta^i=(\theta^i_\ell)_{\ell\in[\pi_\ep]},\> i=1,2).
\]
We denote by $\theta_\ep^*\in\mathrm{Int}(\Theta_\ep)$ the true parameter value.

The observed network process is assumed to satisfy the system of stochastic differential equations
\begin{equation}\label{202508070135}
dX^{\ep,i}_t
=
S^{\ep,i}(t,\bbX^{\ep,i}_t,\theta_\ep^*)\,dt
+
\sigma^{\ep,i}(t)\,dw^i_t,
\qquad
X^{\ep,i}_0=x^{\ep,i}_0,
\qquad i\in[\sfd_\ep],
\end{equation}
where $w^1,\ldots,w^{\sfd_\ep}$ are independent standard Wiener processes,
the diffusion coefficients $\sigma^{\ep,i}$ are progressively measurable,
and the initial values $x^{\ep,i}_0$ are $\calf_0$-measurable random variables.
The drift function $S^{\ep,i}$ may depend on time, on the local state vector
$\bbX_t^{\ep,i}$, and on the parameter $\theta_\ep^*$.

The present formulation is Markovian in the local state vector $\bbX_t^{\ep,i}$.
A path-dependent extension, where the drift depends on the whole trajectory
$\bbX^{\ep,i}=(\bbX_s^{\ep,i})_{s\in[0,t]}$, is also possible, but would require working on an appropriate path space and is not pursued here.

For later estimates, we also define a quantity measuring how often a component appears in the drift equations of other components. For each $i\in[\sfd_\ep]$, set
\[
{\sf o}_\ep^i
=
\sum_{j\in[\sfd_\ep]}1_{\caln_j^\ep}(i),
\]
where $1_{\caln_j^\ep}(i)=1$ if $i\in\caln_j^\ep$ and $0$ otherwise. Thus
${\sf o}_\ep^i$ counts the number of drift equations in which the component $X^{\ep,i}$ appears. We define
\begin{equation}\label{202508070403}
{\sf o}_\ep=\max_{i\in[\sfd_\ep]}{\sf o}_\ep^i,
\qquad
\sfk_\ep=\max_{i\in[\sfd_\ep]}\sfk_i^\ep.
\end{equation}
The quantity $\sfk_\ep$ controls the maximal neighborhood size, while
${\sf o}_\ep$ controls the maximal feedback given by a component across the system.
\begin{example}\rm
A standard network drift decomposes into a momentum effect and network interaction terms:
\begin{equation}\label{202508140256}
S^{\ep,i}(t,\bbX^{\ep,i}_t,\theta)
=
V^{ii}(t,X^{\ep,i}_t,\theta)
+
\sum_{j\in N_i^\ep\setminus\{i\}}
V^{ij}(t,X^{\ep,i}_t,X^{\ep,j}_t,\theta),
\end{equation}
where $N_i^\ep\subset[\sfd_\ep]$ is the neighborhood of node $i$.
In this case,\[
\bbX^{\ep,i}_t=\big(X_t^{\ep,i},(X_t^{\ep,i})_{j\in N^{\ep}_i\setminus\{i\}},(X_t^j)_{j\in N_i\setminus\{i\}}\big).
\]
\end{example}

\begin{example}\rm
A small-interaction model is obtained by scaling the network interaction term:
\begin{equation}\label{202508140257}
S^{\ep,i}(t,\bbX^{\ep,i}_t,\theta)
=
V^{ii}(t,X^{\ep,i}_t,\theta)
+
\ep^\alpha
\sum_{j\in N_i^\ep\setminus\{i\}}
V^{ij}(t,X^{\ep,i}_t,X^{\ep,j}_t,\theta),
\end{equation}
for some $\alpha>0$.
\end{example}

\begin{example}\label{202508140427}\rm
A linearly parametrized principal-interaction model is given by
\begin{equation}\label{202508071013}
S^{\ep,i}(t,\bbX^{\ep,i}_t,\theta_\ep)
=
\sum_{j\in N_i^\ep}
\theta_\ep^{ij}\,
\psi^{ij}(X^{\ep,i}_t,X^{\ep,j}_t)
+
\ep^\alpha
\sum_{j\in N_i^\ep}
V_\ep^{ij}(t,X^{\ep,i}_t,X^{\ep,j}_t).
\end{equation}
Here 
$\bbX^{\ep,i}_t=\big((X^{\ep,i}_t)_{j\in N^\ep_i},(X^{\ep,j}_t)_{j\in N^\ep_i},(X^{\ep,i}_t)_{j\in N^\ep_i},(X^{\ep,j}_t)_{j\in N^\ep_i}\big)$, and 
$\sfk^\ep_i=4|N^\ep_i|$.
Here the coefficients
\[
\theta_\ep=(\theta_\ep^{ij})_{j\in N_i^\ep,\ i\in[\sfd_\ep]}
\]
are the interaction parameters of interest. The functions $\psi^{ij}$ are known feature functions, whereas the functions $V_\ep^{ij}$ represent nuisance components of the drift when they are unknown.
\end{example}

\paragraph{Measurement model and nuisance terms.}

In the linearly parametrized model \eqref{202508071013}, the nuisance functions
$V_\ep^{ij}$ are not assumed to be known or consistently estimable. The inference procedure therefore uses only the parametric part of the drift, encoded by the known functions $\psi^{ij}$, to estimate $\theta_\ep$.

More generally, in this work we allow the true drift functions $S^{\ep,i}$ to contain arbitrary nuisance terms. We assume that inference is based on a \textit{measurement model}, described by functions
\[
\wt S^{\ep,i}:[0,T]\times\bbR^{\sfk_i^\ep}\times\Theta_\ep\to\bbR,
\qquad i\in[\sfd_\ep],
\]
which represent the part of the drift used for estimating the parameter.
For instance, in the case of \eqref{202508071013}, the measurement drift considered is
\begin{equation}\label{202508160503}
\wt S^{\ep,i}(t,\bbz^{\ep,i},\theta)
=
\sum_{j\in\check N_i^\ep}
\theta^{ij}\,
\psi^{ij}(z^{\ep,i}_t,z^{\ep,j}_t).
\end{equation}
Thus the measurement model preserves the known principal interaction structure and omits the nuisance component. We will study this model in detail in \autoref{sec:ident}.

Throughout the paper we assume the following assumptions. 
\begin{enumerate}[label=(A\arabic*)]

\item \label{as:lip} The functions $S^{\ep,i}$ and $\wt{S}^{\ep,i}$ are continuous, and 
there exists a constant $\consta$ such that 
\beas
|\bbS^{\ep,i}(t,x,\theta)|
&\leq& 
\consta \sum_{k=1}^{\sfk_i^\ep}(1+|x_k|)
\eeas
and
\bea\label{202508070344}
|\bbS^{\ep,i}(t,x,\theta)-\bbS^{\ep,i}(t,y,\theta)|
&\leq& 
\consta \sum_{k=1}^{\sfk_i^\ep}|x_k-y_k|
\eea
for all $t\in[0,T]$, $x=(x_k)$, $y=(y_k)\in\bbR^{\sfk_i^\ep}$, $\theta_\ep\in\Theta_\ep$, $i\in[\sfd_\ep]$, and $\bbS^{\ep,i}=S^{\ep,i},\wt{S}^{\ep,i}$. 

\item \label{as:diff}

Suppose that 
\beas 
\big|\sigma^{\ep,i}(t)\big|
&\leq&
\constz\ep\bigg(1+\sum_{k=1}^{\sfk_i^\ep}\big|(\bbX^{\ep,i}_t)_k\big|\bigg).
\eeas
%
a.s. 
for every $i\in[\sfd_\ep]$, $\ep\in(0,1]$. 
\end{enumerate}

For $\ep\in(0,1]$ and $A>0$, let 
\beas 
T(\ep,A) &=& \inf\big\{t\in[0,T];\>\max_{i\in[\sfd_\ep]}|\sigma^{\ep,i}(t)|>A\ep\big\}\wedge T.
\eeas

\ref{as:lip} is a standard assumption concerning Lipschitz continuity and linear growth, with an explicit dependence on $\sfk_i^\ep$. 
\ref{as:diff} allows for a general diffusion term with a mild growth, possibly depending on the neighbors. Our control of the diffusion term will depend on the level crossing probability for $\sigma^{\ep,i}(t)$, expressed in terms of $T(\ep,A)$. This generalizes the classical small noise regime. 
Assumption \ref{as:diff} includes the special case where the diffusion term is bounded by $\epsilon$, 
\begin{equation}\label{eq:bounded-noise}
\max_{i\in[\sfd_\ep]} \sup_{t \in [0, T]} |\sigma^{\ep,i}(t)| \leq \epsilon.
\end{equation}

\section{General error bound for minimum distance estimators} \label{sec:general-err}
The goal of this section is to obtain an error bound for the minimum distance estimator \eqref{eq:md-est} defined below, under minimal assumptions on the data-generating model, allowing for a general nuisance scheme. 
This kind of estimator has been considered in the context of small diffusion problems in \cite{kutoyants1994minimum, kutoyants2012identification}.

We start by describing the estimation
procedure for learning the parameters 
of model \eqref{202508070135}.
The idea is to minimize the distance
between the observed noisy data and 
the corresponding noiseless (ODE) system. 
Moreover, in our setting, the distance is
 measured with respect to the 
measurement model, which neglects the nuisance terms. 
More precisely, 
the measurement model corresponds to the ODE
\bea\label{202508071058}
\frac{d}{dt}z^{\ep,i}_t(\theta_\ep) \yeq \wt{S}^{\ep,i}(t,\bbz^{\ep,i}_t(\theta_\ep),\theta_\ep),\quad z^{\ep,i}_0(\theta_\ep)=x^{\ep,i}_0
\qquad(i\in[\sfd_\ep]).
\eea
Under Condition \ref{as:lip}, the system of equations
has a unique solution $(z^{\ep,i}_t(\theta_\ep))_{i\in[\sfd_\ep],t\in[0,T]}$. 
Each $z^{\ep,i}_t$ is an $\calf_0$-measurable random variable since $x^{\ep,i}_0$ is $\calf_0$-measurable.

Therefore, we introduce the following loss functions: 
\beas 
\calh_\ep^r(\theta_\ep)
&=& 
\sup_{t\in[0,T]}\bigg\{(\sfd_\ep)^{-1}\sum_{i\in[\sfd_\ep]}\big|X^{\ep,i}_t-z^{\ep,i}_t(\theta_\ep)\big|^r\bigg\}^{1/r}\qquad(r\in[1,\infty))
\eeas
and 
\beas 
\calh_\ep^r(\theta_\ep)
&=& 
\sup_{t\in[0,T]}\max_{i\in[\sfd_\ep]}\big|X^{\ep,i}_t-z^{\ep,i}_t(\theta_\ep)\big|\qquad(r=\infty).
\eeas
Since $[0,T]\times\Theta_\ep$ is compact, the map 
$[0,T]\times\Theta_\ep\ni(t,\theta_\ep)\mapsto\big|X^{\ep,i}_t-z^{\ep,i}_t(\theta_\ep)\big|\in\bbR_{\geq0}$ is uniformly continuous, 
and hence the maps $\calh^r_\ep:\Theta_\ep\to\bbR_{\geq0}$ are continuous for $r\in[1,\infty]$.

A measurable map $\wh{\theta}^r_\ep$ is called a \textit{minimum distance} estimator 
for $\calh_\ep^r$ if 
\bea \label{eq:md-est}
\calh_\ep^r\big(\wh{\theta}^r_\ep\big)
&=& 
\inf_{\theta_\ep\in\Theta_\ep}\calh_\ep^r(\theta_\ep). 
\eea

Denote by $\big(x^{\ep,i}_t(\theta_\ep)\big)_{i\in[\sfd_\ep],t\in[0,T]}$ a unique solution of the system of ordinary differential equations 
\beas 
\frac{d}{dt}x^{\ep,i}_t(\theta_\ep)
&=& S^{\ep,i}(t,\bbx^{\ep,i}_t(\theta_\ep),\theta_\ep),\quad x^{\ep,i}_0(\theta_\ep)\yeq x^{\ep,i}_0\qquad(i\in[\sfd_\ep]). 
\eeas
The solution is labeled with $\ep$ since the dimension varies with $\ep$. 
Each map $x^{\ep,i}:[0,T]\times\Theta_\ep\ni(t,\theta_\ep)\mapsto x^{\ep,i}_t(\theta_\ep)\in\bbR$ is continuous. 
Write $x^{\ep,i}_t=x^{\ep,i}_t(\theta_\ep^*)$. 

Let 
\beas 
m_r^\ep(t)
&=& 
\sup_{s\in[0,t]}(\sfd_\ep)^{-1}\sum_{i\in[\sfd_\ep]}\big|X^{\ep,i}_s-x^{\ep,i}_s\big|^r\quad(r\in[1,\infty))
\eeas
and 
\beas 
m_\infty^\ep(t)
&=& 
\sup_{s\in[0,t]}\max_{i\in[\sfd_\ep]}\big|X^{\ep,i}_s-x^{\ep,i}_s\big|.
\eeas
Let 
\beas 
R^\ep_r(t)
&=& 
\sup_{s\in[0,t]}(\sfd_\ep)^{-1}\sum_{i\in[\sfd_\ep]}\bigg|\int_0^s\sigma^{\ep,i}(v)dw^i_v\bigg|^r
\eeas
for $r\in[1,\infty)$, and 
\beas 
R^\ep_\infty(t) 
&=& 
\sup_{s\in[0,t]}\max_{i\in[\sfd_\ep]}\bigg|\int_0^s\sigma^{\ep,i}(v)dw^i_v\bigg|.
\eeas
The following preliminary estimate provides a bound for the distance 
$m^\epsilon_r$ between the observed path and the (noiseless) parametric ODE into a quantity depending on the graph connectivity through $\sfk_\epsilon, \sf{o}_\ep$ and the stochastic remainder $R^\epsilon_r$.

\begin{lemma}\label{lem:ini-bound}
Under \ref{as:lip}, for $t\in[0,T]$, 
\bea\label{202508070505}
m^\ep_r(t)
&\leq&
2^{r-1}R^\ep_r(t)\exp\bigg(2^{r-1}{\consta}^rT^{r-1}t(\max_{i\in[\sfd_\ep]}\sfk^\ep_i)^{r-1}{\sf o}_\ep\bigg)
\eea
for $r\in[1,\infty)$ and 
\bea\label{202508070516}
m_\infty^\ep(t)
&\leq&
R^\ep_\infty(t)\exp\bigg(t\consta\max_{i\in[\sfd_\ep]}\sfk^\ep_i\bigg).
\eea
\end{lemma}
The proof is given in \autoref{sec:proof-ini-bound}.
According to (\ref{202508070505}) and (\ref{202508070516}), 
if ${\sf o}_\ep$  is comparable with $\max_{i\in[\sfd_\ep]}\sfk^\ep_i$, then 
the exponents of $m_1(t)$ and $m_\infty(t)$ are comparable, and 
they are smaller than the exponent of any $m_r(t)$ with $r\in(1,\infty)$. 

The bound of (\ref{202508070505}) depends on the shape of the graph. 
The number ${\sf o}_\ep$ can be large if some $i\in[\sfd_\ep]$ has a large number of outbound edges. 
If some $i\in[\sfd_\ep]$ gives feedback to all other nodes, then ${\sf o}_\ep=\sfd_\ep$. 
The next step is to provide a tail bound for the stochastic remainder term. 

\begin{lemma}\label{lem:stoch-rem-bound}
Under \ref{as:diff}, for $r \in [1, \infty)$
\bea\label{202508070607}
P\big[R^\ep_r(T(\ep,A))^{\frac 1 r}>x\mid\calf_0\big]
&\leq&
P\big[R^\ep_\infty(T(\ep,A))>x\mid\calf_0\big]
\notag 
\\
& \leq &
\constc\sfd_\ep\frac{\sqrt{2TA^2\ep^2}}{x}\exp\bigg(-\frac{x^2}{2TA^2\ep^2}\bigg)\qquad(x>0)
\eea
with $\constc=2\pi^{-1/2}$. 

\end{lemma}
\noindent
The proof is given in \autoref{sec:proof-stoch-rem}.

Next we focus on the error due to 
nuisance and identifiability. 

\paragraph{Nuisance discrepancy.}
The quantity ${m}_r^\ep$ defined 
above referred to the distance between the observable process $(X^{\ep, i}_t, i \in [\sfd_\ep])$ and the ``exact'' ODE $(x^{\ep, i}_t, i \in [\sfd_\ep])$. 
In the case of nuisance terms we have instead access to the measurement 
model $(z^{\ep, i}_t, i \in [\sfd_\ep])$, therefore we 
define $\wt{m}_r^\ep(t)$ and $\wt{m}_\infty^\ep(t)$ by 
\beas 
\wt{m}_r^\ep(t)
&=& 
\sup_{s\in[0,t]}(\sfd_\ep)^{-1}\sum_{i\in[\sfd_\ep]}\big|X^{\ep,i}_s-z^{\ep,i}_s\big|^r\quad(r\in[1,\infty))
\eeas
and 
\beas 
\wt{m}_\infty^\ep(t)
&=& 
\sup_{s\in[0,t]}\max_{i\in[\sfd_\ep]}\big|X^{\ep,i}_s-z^{\ep,i}_s\big|,
\eeas
respectively, where $z^{\ep,i}_s=z^{\ep,i}_s(\theta_\ep^*)$. 
Then, for $r\in[1,\infty)$, we have 
\beas
\big|\wt{m}_r^\ep(t)^{1/r}-m_r^\ep(t)^{1/r}\big|
\yleq
\sup_{s\in[0,t]}\bigg\{(\sfd_\ep)^{-1}\sum_{i\in[\sfd_\ep]}\big|x^{\ep,i}_s-z^{\ep,i}_s\big|^r\bigg\}^{1/r}
\>=:\>\Delta_r(\ep,t),
\eeas
and the \textit{nuisance discrepancy} can be measured as  
\bea\label{eq:delta-r-def}
\Delta_r^\ep=\sup_{t\in[0,T]}\Delta_r(\ep,t). 
\eea
Similarly, 
\beas
\big|\wt{m}_\infty^\ep(t)-m_\infty^\ep(t)\big|
\yleq
\sup_{s\in[0,t]}\max_{i\in[\sfd_\ep]}\big|x^{\ep,i}_s-z^{\ep,i}_s\big|
\>=:\>\Delta_\infty(\ep,t),
\eeas
and we set 
\bea\label{eq:delta-def}
\Delta_\infty^\ep=\Delta_\infty(\ep,T).
\eea

\paragraph{Identifiability gap.}
Let $U_\ep(\nu)=\{\theta_\ep\in\Theta_\ep;\>\rho_\ep(\theta_\ep,\theta_\ep^*)<\nu\big\}$ and $U_\ep(\nu)^c=\Theta_\ep\setminus U_\ep(\nu)$ for $\nu>0$. 
For $r\in[1,\infty]$, define $g_r^\ep(\nu)$ by 
\bea \label{eq:gnu-def}
g_r^\ep(\nu)
&=& 
\inf_{\theta_\ep\in U_\ep(\nu)^c}
\sup_{t\in[0,T]}(\sfd_\ep)^{-1}\sum_{i\in[\sfd_\ep]}\big|z^{\ep,i}_t(\theta_\ep)-z^{\ep,i}_t(\theta_\ep^*)\big|^r\quad(r\in[1,\infty))
\eea
and 
\bea \label{eq:gnu-inf-def}
g_\infty^\ep(\nu)
&=& 
\inf_{\theta_\ep\in U_\ep(\nu)^c}
\sup_{t\in[0,T]}\max_{i\in[\sfd_\ep]}\big|z^{\ep,i}_t(\theta_\ep)-z^{\ep,i}_t(\theta_\ep^*)\big|.
\eea
{We remark that $g_r^\ep(\nu)$ and $g_\infty^\ep(\nu)$ are $\calf_0$-measurable random variables. }%
These terms are standard in the literature. 
They represent an \textit{identifiability gap}.
If $g_r^\ep(\nu)$ is small for large $\nu$, it means that our measurement model cannot effectively
distinguish between different parameters. 

In preparation for the main result of this section we introduce the following quantities. 
For $r\in[1,\infty)$ and $A>0$, define an $\calf_0$-measurable random variable $\Xi_r^\ep(\nu,A)$ as 
\bea\label{202508070650}
\Xi_r^\ep(\nu,A)
&=&
2\constn A^{-1}\ep^{-1}(g^\ep_r(\nu)^{1/r}-2\Delta_r^\ep)\exp\bigg(-\consto(\max_{i\in[\sfd_\ep]}\sfk^\ep_i)^{r-1}{\sf o}_\ep\bigg)
\eea
with $\constn=2^{-\frac{7}{2}+\frac{1}{r}}T^{-\half}$ and $\consto=2^{r-1}r^{-1}{\consta}^rT^r$. 
Similarly, for $A>0$, 
define an $\calf_0$-measurable random variable $\Xi_\infty^\ep(\nu,A)$ as 
\bea\label{202508070739} 
\Xi_\infty^\ep(\nu,A)
&=& 
\constp A^{-1}\ep^{-1}(g^\ep_\infty(\nu)-2\Delta^\ep_\infty)\exp\bigg(-\constq\max_{i\in[\sfd_\ep]}\sfk^\ep_i\bigg)
\eea
with $\constp=2^{-\frac{3}{2}}T^{-\half}$ and $\constq=T\consta$.


The following theorem is the main generic finite-sample result of the paper. It applies to a broad class of minimum-distance estimators
 $\wh{\theta}_\ep^r$, $r\in[1,\infty]$ under nuisance misspecification and weak identifiability control.

\begin{theorem}\label{thm:err-bound}
Suppose that \ref{as:lip} and \ref{as:diff} are satisfied. Then 
\bea\label{202508140415}
P\big[\rho_\ep(\wh{\theta}_\ep^r,\theta_\ep^*)\geq\nu\big]
&\leq&
2\constc\sfd_\ep e^{-\xi^2}+P\big[\Xi^\ep_r(\nu,A)<\xi\big]+P\big[T(\ep,A)<T\big]
\eea
for all $\ep\in(0,1]$, $\nu>0$, $\xi>0$, $A>0$ and $r\in[1,\infty]$. 
\end{theorem}
The proof is given in \autoref{sec:proof-general-err}.

Estimating $g^\ep_r(\nu)$ from below for $r\in[2,\infty]$ is crucial in applications of Theorem \ref{thm:err-bound}. We proceed as follows
\bea\label{202508071017}
g_\infty^\ep(\nu)^2
&\geq& 
g^\ep_r(\nu)^{2/r} \ygeq g^\ep_2(\nu)
\nn\\&=&
\inf_{\theta_\ep\in U_\ep(\nu)^c}
\sup_{t\in[0,T]}(\sfd_\ep)^{-1}\sum_{i\in[\sfd_\ep]}\big|z^{\ep,i}_t(\theta_\ep)-z^{\ep,i}_t(\theta_\ep^*)\big|^2
\nn\\&\geq&
\inf_{\theta_\ep\in U_\ep(\nu)^c}
T^{-1}\int_0^T(\sfd_\ep)^{-1}\sum_{i\in[\sfd_\ep]}\big|{ z}^{\ep,i}_t(\theta_\ep)-{ z}^{\ep,i}_t(\theta_\ep^*)\big|^2dt,
\eea
However, obtaining a meaningful lower bound from the final expression requires additional information on the structure of $z^{\epsilon,i}_t$. We address this issue in the following section.

\section{Identifiability and consistency under repeated measurements }\label{sec:ident}

The main challenge in applying \autoref{thm:err-bound} in its general form is to control the error components due to the presence of the nuisance function and the possible lack of identifiability. 
In this section we focus on these issues. 
First we highlight the need for repeated measurements: as we show in \autoref{sec:repeated}, even simpler models may fail to be identifiable. In order to circumvent this problem we introduce a repeated measurements scheme. 
Then, in \autoref{sec:estimation-inter}, we focus on a specific model form, which allows us to obtain a non-asymptotic error bound explicitly depending on $\epsilon$, the graph size, and the model constants.
Finally, in 
\autoref{sec:cons-balance}
we discuss consistency and balance conditions.

\subsection{Repeated measurements}\label{sec:repeated}
We consider the system of stochastic differential equations having the same form as (\ref{202508070135}): 
\bea\label{202508170138}
dX^{\ep,i}_t &=& S^{\ep,i}(t,\bbX^{\ep,i}_t,\theta_\ep)dt+\sigma^{\ep,i}(t)dw^i_t, 
\quad
X^{\ep,i}_0\yeq x^{\ep,i}_0
\quad(i\in[\sfp_\ep]). 
\eea
The index $i\in[\sfp_\ep]$, not $i\in[\sfd_\ep]$, and $\bbX^{\ep,i}_t=(X^k_t)_{k\in\caln_i^\ep}$ with 
$\caln_i^\ep\in[\sfp_\ep]^{\sfk^\ep_i}$. 
Instead, we assume that $n$ repeated samples of $(X^{\ep,i})_{i\in[\sfp_\ep]}$ are available on $(\Omega,\calf,\bbF,P)$. 
Then the whole observations are described by the model (\ref{202508070135}) with $\sfd_\ep=\sfp_\ep n$, $i\in[\sfd_\ep]$, 
$
S^{\ep,i}\yeq S^{\ep,r(i)}
$
for $r(i)=\sfp_\ep+i-\sfp_\ep\lceil i/\sfp_\ep\rceil $, and 
$\bbX^{\ep,i}_t=(X^k_t)_{k\in\caln_i^\ep}$ for $\caln_i^\ep=\caln_{r(i)}^\ep+\sfp_\ep\lfloor(i-1)/\sfp_\ep\rfloor$, where 
$x+c=(x_k+c)$ for a vector $x=(x_k)$ and a number $c$. 
The $n$ groups of data can be dependent.

Given functions $\wt{S}^{\ep,i}$ ($i\in[\sfp_\ep]$), we set $\wt{S}^{\ep,i}=\wt{S}^{\ep,r(i)}$ for $i\in[\sfd_\ep]$. 
We consider the estimators $\wh{\theta}_\ep^r$ ($r\in[1,\infty]$) based on the observation of the $\sfd_\ep$-dimensional process $(X^{\ep,i})_{i\in[\sfd_\ep]}$ as in Section \ref{sec:general-err}. 
Then Theorem \ref{thm:err-bound} is valid, that is, the inequality (\ref{202508140415}) holds 
under Conditions \ref{as:lip} and \ref{as:diff}. 


Set $\cali^\ep_m=[m\sfp_\ep]\setminus[(m-1)\sfp_\ep]$ for $m\in[n]$. 
For example, the $n$ $\sfp_\ep$-dimensional random processes $(X^{\ep,i})_{i\in\cali^\ep_m}$ ($m\in[n]$) 
form an i.i.d. sequence 
if the variables $\big(x^{\ep,i}_0,\sigma^{\ep,i},w^{\ep,i}\big)_{i\in\cali^\ep_m}$ ($m\in[n]$) 
are $n$ independent copies of $\big(x^{\ep,i},\sigma^{\ep,i},w^{\ep,i}\big)_{i\in[\sfp_\ep]}$. 
Conditions \ref{as:lip} and \ref{as:diff} are requesting their inequalities for $i\in[\sfd_\ep]$, which is equivalent to requiring them for $i\in[\sfp_\ep]$, in this situation. 
In Section \ref{sec:estimation-inter}, we will treat a more general case where the Wiener processes $\big(w^{\ep,i}\big)_{i\in[\sfd_\ep]}$ are independent, but 
the $n$ variables $\big(x^{\ep,i}_0\big)_{i\in\cali^\ep_m}$ ($m\in[n]$) are weakly dependent 
and $\big(\sigma^{\ep,i}\big)_{i\in[\sfd_\ep]}$ are only controlled from above. 
The weak dependency assumption is useful for instance to treat $n$ blocks of data repeatedly sampled from a single $\sfp_\ep$-dimensional time series at certain intervals. 

The following example highlights the need for repeated measurements. 

 \begin{example}
Let us consider the model (\ref{202508170138}) with $S^{\ep,i}(t,\bbX^{\ep,i}_t,\theta_\ep)=\theta_{\ep,i}X^{\ep,i}_t$ and $\sigma^{\ep,i}(t)=1$. 
The random variables $x^{\ep,i}_0$ ($i\in[\sfp_\ep]$) are i.i.d. $\sim \text{Bernoulli}(1,1/2)$. 
Suppose that $\Theta_\ep=[0,1]^{\sfp_\ep}$ and 
$\rho_\ep(\theta_\ep,\theta_\ep')=\big\{(\sfp_\ep)^{-1}\sum_{k\in[\sfp_\ep]}|\theta_{\ep,k}-\theta'_{\ep,k}|^p\big\}^{1/p}$ for $\theta_\ep=(\theta_{\ep,k})_{k\in[\sfp_\ep]}$ and $\theta'_\ep=(\theta'_{\ep,k})_{k\in[\sfp_\ep]}$, where $p$ is an arbitrary number in $[1,\infty)$. 
Let $\theta_{\ep,i}^*=1/2$. 
For $r\in[1,\infty)$, the contrast function $\calh^r_\ep$ is 
\beas 
\calh^r_\ep(\theta_\ep) 
&=& 
\sup_{t\in[0,T]}\bigg\{(\sfp_\ep)^{-1}\sum_{i\in[\sfp_\ep]:x^{\ep,i}_0=1}\big|X^{\ep,i}_t-e^{\theta_{\ep,i}t}\big|^r
+(\sfp_\ep)^{-1}\sum_{i\in[\sfp_\ep]:x^{\ep,i}_0=0}\big|X^{\ep,i}_t\big|^r
\bigg\}^{1/r}
\eeas
since 
$z^{\ep,i}_t(\theta_\ep)=x^{\ep,i}_0e^{\theta_\ep^it}=0$ when $x^{\ep,i}_0=0$. 
There is a minimizer $\wh{\theta}_\ep^r=(\wh{\theta}_{\ep,i})_{i\in[\sfp_\ep]}$ of $\calh^r_\ep$ such that 
$\wh{\theta}_{\ep,i}^r=0$ for all $i\in[\sfp_\ep]$ whenever $x^{\ep,i}_0=0$. 
Then
\beas 
\inf_{\ep\in(0,1]}P\bigg[(\sfp_\ep)^{-1}\sum_{i\in[\sfp_\ep]}\big|\wh{\theta}_{\ep,i}^r-\theta_{\ep,i}^*\big|^p\geq2^{-p-1}\bigg]
\ygeq
\inf_{\ep\in(0,1]}P\bigg[(\sfp_\ep)^{-1}\sum_{i\in[\sfp_\ep]}1_{\{x^{\ep,i}_0=0\}}\geq2^{-1}\bigg]
\geq\half
\eeas
regardless of the behavior of $\sfp_\ep$. 
The estimator $\wh{\theta}_\ep^r$ is inconsistent in $\rho_\ep$-metric for any $p\in[1,\infty)$, 
hence 
inconsistent in the metric $\rho_\ep(\theta_\ep,\theta_\ep')=\max_{i\in[\sfp_\ep]}\big|\theta_{\ep,i}-\theta'_{\ep,i}\big|$. 
 \end{example}
 
This example 
suggests that some restrictive identifiability condition (like the location of the initial values) may be necessary for consistency of estimators in experiments without repeated measurements, which we do not pursue here. 

\subsection{Estimation of the graph interactions}\label{sec:estimation-inter}
\begin{comment}
{\bf \colorr Under construction. This attempt should be valid for $\theta_\ep$ that has a lower dimension of $\sfd_\ep$ (the dimension of the diffusion process). At present, the setting that assumes only one $i$-th observation $x^{\ep,i}_0$ for the parameters $\theta^{ij}$ is a similar setting like the Neyman-Scott problem. We cannot get a separation of the parametric model. from the data $x^{\ep,i}_0$.

Do it for repeated measurements with stationary mixing initial values. }
\end{comment}
\begin{comment}
{\gray
We will treat a nearly linear model. 
Consider the system of stochastic differential equations
\bea\label{202508070752}
dX^{\ep,i}_t &=& S^{\ep,i}\big((1_{\{j\in\caln^\ep_i\}}X^j_t)_{j\in[\sfd_\ep]},(\theta^{ij})_{j\in\caln^\ep_i}\big)dt+\ep dw^i_t, 
\quad
X^{\ep,i}_0\yeq x^{\ep,i}_0
\quad(i\in[\sfd_\ep])
\eea
The parameter is $\theta_\ep=(\theta^{ij})_{j\in\caln^\ep_i,i\in[\sfd_\ep]}$. 
}
\end{comment}

%
To demonstrate an application of Theorem \ref{thm:err-bound}, from now on
we focus on a linearly parametrized principal interaction model in a repeated measurements setting. 
This allows us in particular to give an explicit characterization of the error component due to nuisance and identifiability.
Specifically, 
we consider $n$ observations of the data generated by the model (\ref{202508170138}) with the drift term (\ref{202508071013}): 
\bea \label{eq:interactions-model}
S^{\ep,i}(t,\bbX^{\ep,i}_t,\theta_\ep)
&=& 
\sum_{j\in\checkN_i^\ep}
\theta^{ij}_\ep \psi^{ij}\big(X^{\ep,i}_t,X^{\ep,j}_t\big)
+\ep^\alpha\sum_{j\in\checkN_i^\ep}V^{ij}_\ep(t,X^{\ep,i}_t,X^{\ep, j}_t\big),
\eea 
where $\theta_\ep=(\theta^{ij}_\ep)_{j\in N_i^\ep,i\in[\sfp_\ep]}$. 
$N^\ep_i$ is the neighborhood of $i$, possibly having $i$ as an element. 

Here $\psi^{ij}$ are known but $V^{ij}_\ep$ are assumed to be unknown. 
We do not need to know the value of the positive parameter $\alpha$, however, the convergence of the estimator is affected by its value. 
As mentioned above, the total number of components is $\sfd_\ep=n\sfp_\ep$ and the components are arranged in the cyclic order: 
$\psi^{ij}\yeq \psi^{r(i),j}$ 
for $r(i)=\sfp_\ep+i-\sfp_\ep\lceil i/\sfp_\ep\rceil $, 
$V^{ij}_\ep=V^{r(i)j}_\ep$ 
and 
$\bbX^{\ep,i}_t=(X^k_t)_{k\in\caln_i^\ep}$ for 
\bea\label{202508210319}
\caln_i^\ep&=&\caln_{r(i)}^\ep+u_\ep(i)\quad\text{with}\quad u_\ep(i)=\sfp_\ep\lfloor(i-1)/\sfp_\ep\rfloor. 
\eea

In words, we regard the \(n\) repeated observations as one enlarged
system of dimension \(\sfd_\ep=n\sfp_\ep\). Each block contains one copy
of the original \(\sfp_\ep\)-dimensional network, and the maps \(r(i)\)
and \(u_\ep(i)\) simply identify, respectively, the node label within the
block and the block offset.

The functions $\wt{S}^{\ep,i}(t,\bbz^{\ep,i},\theta)$ are defined by (\ref{202508160503}) for $i\in[\sfp_\ep]$ and 
extended as $\wt{S}^{\ep,i}=\wt{S}^{\ep,r(i)}$ for $i\in[\sfd_\ep]$. 
The $\calf_0$-measurable initial values $x^{\ep,i}_0$ are sampled $\sfd_\ep$ times. 
The variables $\{x^{\ep,i}_0\}_{i\in[\sfd_\ep]}$ are dependent in general. 
The family $\{w^i\}_{i\in[\sfd_\ep]}$ is a collection of independent standard Wiener processes 
for $w^i=(w^i_t)_{t\in[0,T]}$. 
%
In the repeated measurements setting, the definition (\ref{202508070403}) of ${\sf o}_\ep$ and
$\sfk_\epsilon$ are rewritten as
\[
{\sf o}_\ep = \max_{i\in[\sfp_\ep]}\sum_{j\in[\sfp_\ep]}1_{\caln^\ep_{j}}(i)
\qquad 
\sfk_\epsilon = \max_{i\in[\sfp_\ep]}\sfk^\ep_i.
\]
In the following, the metric $\rho_\ep$ on the parameter space $\Theta_\epsilon$ is set as 
\bea\label{202508171526}
\rho_\ep(\theta_\ep,\theta_\ep')
&=& 
\bigg\{(\sfp_\ep)^{-1}\sum_{i\in[\sfp_\ep]}
\sum_{j\in\checkN_i^\ep}
(\theta^{ij}-\theta^{ij'})^2\bigg\}^{1/2}. 
\eea
\noindent

\begin{comment}
{\gray
\xout{
Remark that $\theta^{ij}$ and $\theta^{ji}$ are indistinguishable if $\psi^{ij}(x,y)=\psi^{ji}(y,x)$. 
Identifiability is a delicate matter. }
}
\end{comment}
%
\begin{comment}
{\gray
$x^{\ep,i}_t(\theta_\ep)$ is demonized but it keeps the $\ep^\alpha$-order term. 
However, the effect of the $\ep^\alpha$-order term to the non-degeneracy is vanishing as $\ep\down0$. 
}
\end{comment}
%
%
%
We consider the following refinement of \ref{as:lip} adapted to the present model \eqref{eq:interactions-model}.

\begin{enumerate}[label=(A\arabic*$'$)]
\item \label{as:drift-lin} 
\begin{enumerate}[label=(\roman*)]
\item 
There exists a constant $\consti$ such that 
\beas
|\psi^{ij}(x)|\vee|V^{ij}_\ep(t,x,\tau_\ep)|
&\leq& 
\consti (1+ |x|_\infty)
\eeas
and
\beas
|\psi^{ij}(x)-\psi^{ij}(y)|
\vee|V^{ij}_{\ep}(t,x,\tau_\ep)-V^{ij}_{\ep}(t,y,\tau_\ep)|
&\leq& 
\consti |x-y|_\infty
\eeas
for all $t\in[0,T],\>x,y\in\bbR^2,\>j\in N^\ep_i,\>i\in[\sfp_\ep],\>\ep\in(0,1]$, and 
$|x|_{\infty} =  \max_{i = 1,2} |x_i|, x \in \bbR^2$.
\item  $\ds\constt:=\esssup_{\omega\in\Omega}\sup_{\ep\in(0,1]}\sup_{i\in[\sfd_\ep]}|x^{\ep,i}_0|<\infty$. 
\item $\ds\constj:=\sup_{\ep\in(0,1]}\sup_{i\in[\sfp_\ep]}\sup_{j\in N^\ep_i}\sup_{\theta_\ep=(\theta_\ep^{ij})\in\Theta_\ep}
|\theta^{ij}_{\ep}|<\infty$. 
\end{enumerate}
\end{enumerate}
Furthermore, 
we adopt the following condition to ensure the identifiability of the model. 
\begin{enumerate}[label=(L)]
\item \label{as:ident}
There exists a sub $\sigma$-field $\calb^\ep_i$ of $\calf$ for each $\ep\in(0,1]$ and $i\in[\sfd_\ep]$ such that 
\beas 
\inf_{i\in[\sfd_\ep]}\lambda_{\min}
\bigg(E\bigg[\text{Cov}_{\calb^\ep_i}\big[\psi^{ij}\big(x^{\ep,i}_0,x^{\ep,j}_0\big),
\psi^{ij'}\big(x^{\ep,i}_0,x^{\ep,j'}_0\big)\big]\bigg]\bigg)_{j,j'\in N^\ep_i}
&\geq&
\constk\ep^\beta
\quad(\ep\in(0,1])
\eeas
for some constants $\beta\geq0$ and $\constk>0$, 
where $\text{Cov}_{\calb^\ep_i}$ denotes the conditional covariance given $\calb^\ep_i$, and $\lambda_{\min}$ denotes the minimum eigenvalue of the matrix. 
\end{enumerate}

We derive a concentration inequality for $T(\ep,A_\ep)$. We do so by applying arguments similar to the estimates leading to (\ref{202508070505}).
\begin{lemma}\label{lem:diff-bound}
Under \ref{as:drift-lin} and \ref{as:diff}
\beas 
P\big[T(\ep,A_\ep)<T\big]
&\leq&
\sfd_\ep\times
2\pi^{-1/2}\frac{\sqrt{2T A_\ep^2\ep^2}}{x}\exp\bigg(-\frac{x^2}{2TA_\ep^2\ep^2}\bigg)
\eeas
whenever 
\bea \label{eq:a-cond}
A_\ep>2\constz\bigg[1
+(\max_{i\in[\sfd_\ep]}\sfk^\ep_i)\sup_{i\in[\sfd_\ep]}\sup_{s\in[0,T]}\big|x^{\ep,i}_s\big|\bigg], 
\eea
where
\beas 
x &=& 
2^{-1}\constz^{-1}A_\ep(\max_{i\in[\sfd_\ep]}\sfk^\ep_i)^{-1}\exp\bigg(-\consta T\max_{i\in[\sfd_\ep]}\sfk^\ep_i\bigg) .
\eeas
\end{lemma}
\noindent
The proof is given in \autoref{sec:proof-diff-bound}.

The main technical challenge for this section is to obtain a control on 
the \textit{nuisance discrepancy} terms 
\eqref{eq:delta-r-def}-\eqref{eq:delta-def} and 
\textit{identifiability gap}
\eqref{eq:gnu-def}-\eqref{eq:gnu-inf-def}.

As for $\Delta^\ep_r$ and 
$\Delta^\ep_\infty$, defined by \eqref{eq:delta-r-def}-\eqref{eq:delta-def}, we have the following bound. 
\begin{lemma}\label{lem:delta}
    It holds that, under \ref{as:drift-lin}, 
    \bea\label{eq:delta-inf-bound}
    \Delta^\ep_r \leq \Delta^\ep_\infty
     &\leq&
    \consti\ep^\alpha\sfk_\ep T
    \bigg\{1+ \constj
    +
    T\consti \sfk_\ep\bigg\}
    \exp\big(T(\consti+\consti\constj) \sfk_\ep\big). 
    \eea
\end{lemma}
\noindent
The proof is given in \autoref{sec:proof-delta}.

In the linearly parametrized model, the nuisance discrepancy can be controlled by means of simpler, more explicit quantities. Namely, let
\bea\label{202508171722}
Q^\ep(\theta_\ep)
&=& 
3^{-1}T^{-1}(\sfd_\ep)^{-1}\sum_{i\in[\sfd_\ep]}
\big|\sum_{j\in\checkN_i^\ep}(\theta^{ij}_\ep-\theta^{ij*}_\ep) \psi^{ij}\big(x^{\ep,i}_0,x^{\ep,j}_0\big)\big|^2,
\eea
where 
\bea\label{202508202034}
\theta_\ep^{ij}&=&\theta_\ep^{r(i)r(j)}\quad(i,j\in[\sfd_\ep]). 
\eea
%
The following Lemma shows that 
the quantities $g_r^\ep(\nu), g_\infty^\ep(\nu)$ 
\eqref{eq:gnu-def}-\eqref{eq:gnu-inf-def}
can indeed be controlled in terms of $Q^\epsilon$
and a term depending on the graph connectivity. 

\begin{lemma}\label{lem:ident-gap-1} 
Suppose that Condition \ref{as:drift-lin} is satisfied. 
Let $r\in[2,\infty)$ and $\constd>0$. Then 
\bea\label{202508100930}
g_\infty^\ep(\nu)^2
\ygeq 
g_r^\ep(\nu)^{2/r}
\ygeq
g_2^\ep(\nu)
\ygeq
\inf_{\theta_\ep\in U_\ep(\nu)^c}G^\ep(\theta_\ep,\constd)-I^\ep(\constd)
\eea
for $\ep\in(0,T^{3/\constd}\wedge1]$ and $\nu>0$, 
where  
\bea\label{eq:i-def}
G^\ep(\theta_\ep,\constd)\yeq\ep^\constd Q^\ep(\theta_\ep)\quad\text{and}\quad
I^\ep(\constd)
\yeq
\constl
(\max_{i}\sfk^\ep_i)^3\ep^{\frac{4}{3}\constd}
\exp\bigg(\constm\ep^{\constd/3}(\max_{i\in[\sfd_\epsilon]}\sfk^\ep_i)\bigg)
\eea
with
$\constm=2\consti\constj$ and 
\beas 
\constl
&=& 
\frac{3}{2}{\consti^3}(1+\constt)^2{(\constj)^3}. 
\eeas
\end{lemma}
\noindent
The proof is given in \autoref{sec:proof-ident-1}.

%

%
%
Now let 
\bea\label{202508171719}
\varphi_\ep
&=& 
\sup_{\theta_\ep\in\Theta_\ep}\big|Q^\ep(\theta_\ep)-E\big[Q^\ep(\theta_\ep)\big]\big|
\eea
and define the event
\bea\label{202508210601}
\Omega^\ep_0(\nu)=\big\{\varphi_\ep\leq2^{-1}\inf_{\theta_\ep\in U_\ep(\nu)^c}E\big[Q^\ep(\theta_\ep)\big]\big\}. 
\eea
The dependence on $Q^\epsilon$ of the forthcoming error bound will be in terms of the probability of the event $\Omega^\ep_0(\nu)$. 

Thanks to the setting above,
the quantities 
$\Xi_r^\ep(\nu,A), \Xi_\infty^\ep(\nu,A)$
in \eqref{202508070650} - \eqref{202508070739}, 
can be replaced by explicit terms that do not depend on the identifiability gap. For any sequences $(A_\ep)_{\ep\in(0, 1]}$ and $(\nu_\ep)_{\ep\in (0, 1]}$ of positive numbers, 
define $\xi^\ep_r$ as 
\bea \label{eq:xir-def} 
\xi^\ep_r
&=& 
\constr\constn A_\ep^{-1}\ep^{-1+(\constv)/2}\nu_\ep\exp\bigg(-\consto (\max_{i\in[\sfp_\ep]}\sfk^\ep_i)^{r-1}{\sf o}_\ep\bigg)
\eea
for $r\in[2,\infty)$. 
Define $\xi^\ep_\infty$ by 
\bea \label{eq:xi-inf-def}
\xi^\ep_\infty
&=& 
2^{-1}\constr\constp A_\ep^{-1}\ep^{-1+(\constv)/2}\nu_\ep\exp\bigg(-\constq \max_{i\in[\sfp_\ep]}\sfk^\ep_i\bigg).
\eea
%
%

%

Let 
\beas 
\wt{\psi}^{ij}
&=& 
\psi^{ij}\big(x^{\ep,i}_0,x^{\ep,j}_0\big)-E_{\calb^\ep_i}\big[\psi^{ij}\big(x^{\ep,i}_0,x^{\ep,j}_0\big)\big].  
\eeas
The following condition is a sufficient condition for \ref{as:ident}. 
\begin{enumerate}[label=(L$^\sharp$)]

\item \label{as:ident-suff}
For each $\ep\in(0,1]$ and $i\in[\sfd_\ep]$, there exists a sub $\sigma$-field $\calb^\ep_i$ of $\calf$ such that 
\begin{enumerate}[label=(\roman*)]
\item 
$\ds
E\big[\wt{\psi}^{ij}\wt{\psi}^{ij'}\big]
=0
$
for all $j,j'\in N^\ep_i$ with $j\not=j'$. 
\item 
$\ds
\inf\bigg\{E\big[(\wt{\psi}^{ij})^2\big];\>
j\in\checkN_i^\ep,\>
i\in[\sfd_\ep],\>\ep\in(0,1]\bigg\}
>
0$. 
\end{enumerate}
\end{enumerate}

The orthogonality condition (i) of \ref{as:ident-suff} is satisfied if $\big\{\psi^{ij}(x^{\ep,i}_0,x^{\ep,j}_0)\big\}_{j\in N^\ep_i}$
are $\calb^\ep_i$-conditionally independent. 
The $\sigma$-field $\calb^\ep_i$ is arbitrary. It can be the trivial $\sigma$-field. 
An example of $\calb^\ep_i$ is the $\sigma$-field $\sigma[x^{\ep,i}_0]$ generated by $x^{\ep,i}_0$.

We now specialize the general error bound of Theorem \ref{thm:err-bound} to the repeated-measurements linear interaction model.

\begin{theorem}\label{thm:est-bound}
Suppose that Conditions \ref{as:drift-lin}, \ref{as:diff} and \ref{as:ident} are fulfilled. 
Let $r\in[2,\infty]$, $A_\ep > 0$ and $\nu_\ep > 0$. 
Then, for every $\ep\in(0,T^{3/\constd}\wedge1]$,
\bea\label{202508162002}
P\big[\rho_\ep(\wh{\theta}_\ep^r,\theta_\ep^*)\geq\nu_\ep\big]
&\leq&
2\constc\sfd_\ep e^{-(\xi^\ep_r)^2}+P\big[(\Omega^\ep_0(\nu_\ep))^c\big]+P\big[T(\ep,A_\ep)<T\big]
\nn\\&&
+1_{\big\{
\constr\ep^{(\constv)/2}\nu_\ep<(I^\ep(\constd)^{1/2})\vee(4\Delta^\ep_r)\big\}}.
\eea

\end{theorem}

The quantities $A_\ep$ and $\nu_\ep$ are allowed to depend on $\ep$. 
No additional restriction on these sequences is needed for the present finite-sample bound. 
In the next section, we will specify suitable choices of $A_\ep$ and $\nu_\ep$ in order to derive consistency and convergence rates.

\noindent
The proof is given in \autoref{sec:proof-est-bound}. 


\begin{remark}
The terms appearing in the bound of Theorem \ref{thm:est-bound} can be controlled explicitly.
More precisely, the quantity $\xi_r^\ep$ is given by \eqref{eq:xir-def}--\eqref{eq:xi-inf-def}, the nuisance term $\Delta_r^\ep$ is bounded in Lemma \ref{lem:delta}, and the diffusion-threshold term $P\big[T(\ep,A_\ep)<T\big]$ is controlled in Lemma \ref{lem:diff-bound}.
The remaining term $P\big[(\Omega^\ep_0(\nu_\ep))^c\big]$ requires an additional mixing condition, and is controlled in Lemma \ref{lem:omega-bound}.
\end{remark}

\subsection{Consistency and convergence rates}\label{sec:cons-balance}

We are now ready to discuss the consistency of the minimum distance estimator 
\eqref{eq:md-est} starting from the bound in \autoref{thm:est-bound}. This result requires a certain mixing condition. Moreover, it is crucial to identify suitable scaling conditions that control the growth of the graph in relation with the noise level. 
In the following we work in the case $r \in [2, \infty)$. 
See Remark \ref{rem:l-inf} for a discussion of the case $r = \infty$. 

Specifically, we assume the following $\alpha$-mixing condition. 
\begin{enumerate}[label=(B\arabic*)]
    \item
    \label{as:mixing}
    Let $y^\ep_m=(x^{\ep,i}_0)_{i\in\cali^\ep_m}$ for $m\in[n]$. 
Let 
$\calc_m^\ep=\sigma\big[x^{\ep,i}_0;\>i\in\cali_m^\ep\big]$. 
Define $\alpha^\ep(h)$ for $h\in\bbN$ as 
\beas 
\alpha^\ep(h) 
&=& 
\sup\bigg\{\big|P[{\tt A}\cap {\tt B}]-P[{\tt A}]P[{\tt B}]\big|;\>
{\tt A}\in\calc_{\tt a}^\ep,\>
{\tt B}\in\calc_{\tt b}^\ep,\>{\tt a}, {\tt b}\in[n],\>{\tt b}-{\tt a}\geq h\bigg\}
\eeas
for $h=0,...,n-1$, and $\alpha^{\ep}(h)=0$ for $h\geq n$, and suppose that 
\[
\sum_{h=0}^{\infty} \left(1 - \frac hn \right) \alpha^\ep(h) < \infty. 
\]

\end{enumerate}

We remark that Assumption \ref{as:mixing} is weaker than the usual definition assuming 
${\tt A}\in\bigvee_{m\in[n]:m\leq{\tt a}}\calc_m^\ep$ and ${\tt B}\in\bigvee_{m\in[n]:m\geq{\tt b}}\calc_m^\ep$. 

This condition is needed for controlling the 
probability of the event $\Omega_0^\epsilon$, 
as the next Lemma shows. 
\begin{lemma}\label{lem:omega-bound}
Suppose that Conditions \ref{as:drift-lin}, \ref{as:mixing} and \ref{as:ident} are fulfilled. Then 
\begin{align}\label{202508210651}
P\big[(\Omega^\ep_0(\nu))^c\big]
&\leq
{16\sqrt{2}}(\constj)^2\constu\constk^{-1}\>
\bigg\{\sum_{h=0}^{n-1}(1-n^{-1}h)\alpha^\ep(h)\bigg\}^{1/2}
\epsilon^{-\beta}
\nu^{-2}
(\max_{i\in[\sfp_\ep]}\sfk^\ep_i)^2\>
n^{-1/2}
\end{align}
for $\nu>0$, $\ep\in(0,1]$ and $n\in\bbN$. 
\end{lemma}
\noindent The proof is given in \autoref{sec:proof-omega}.

Moreover, we introduce the following balance conditions.

\begin{enumerate}[label=(B\arabic*), resume]
\item \label{as:balance}
Given positive constants $\lS, c_1, \alpha$ and $\beta \geq 0$ such that
\begin{enumerate}[label=(\roman*)]
    \item 
        \begin{gather}
            c_1 + \beta < 2   \label{eq:bc1}\\
            c_1 > 3 \beta \label{eq:bc2} \\
            \lS + \frac{\constv} 2 < \alpha \wedge 1 \label{eq:bc3},
        \end{gather}
    assume the following balance conditions:
    \item (Dimension scaling) 
    for some $\lG>0$, 
    \[\sfp_\epsilon = \mathcal O\left( \epsilon^{-\lG} \right). \]

    \item (Graph sparsity scaling) 
    \bea \label{eq:bal-ko}
    \exp\{(C_{3,r}\vee 1) \sfk^{r-1}_\epsilon \mathsf{o}_\epsilon \} = \mathcal O \left (\epsilon^{-\lS} \right). \quad 
    \eea

    \item (Sample size growth)
    for some $\lN > 0$,
    \bea \label{eq:bal-n}
    {\sqrt n_\ep} =  \mathit{\Theta} \left(\ep^{-\beta - \lN} \right) 
    \eea

    \item (Noise threshold)
    for some $\lA > 0$,
    \[ A_\epsilon =  \mathit{\Theta} \left( \sfk_\epsilon^{1 + \lA} \right).
    \footnote{We write $f(\ep)= \mathit{\Theta}(g(\ep))$ if there exist $m, M > 0$ such that $m g(\epsilon) \leq f(\epsilon) \leq M g(\epsilon) $ for all $\epsilon < \epsilon_0$.} 
    \]

\end{enumerate}
\end{enumerate}

The graph sparsity scaling condition intuitively represents a  logarithmic growth of the graph connectivity with respect to the number of nodes, namely,
    \beas
     \sfk^{r-1}_\epsilon \mathsf{o}_\epsilon = \mathcal O( \log \sfp_\epsilon).
    \eeas
Explicit scalings in terms of $\ep$ can be obtained by combining the preceding conditions with the dimension scaling. 
For instance, one admissible choice is
\bea\label{eq:bal-ko2}
\sfk_\epsilon = \mathcal O\left( \left(\log \left(\frac 1 \ep \right)\right) ^{\frac 1{2(r-1)}}\right),
\qquad 
\mathsf{o}_\epsilon = \mathcal O\left( \left(\log \left(\frac 1 \ep \right) \right)^{\frac 1{2}}\right).
\eea
The particular exponents in \eqref{eq:bal-ko2} are not essential; they are chosen only to provide a concrete admissible scaling and to simplify the notation in the proof. 
More generally, any growth regime satisfying \eqref{eq:bal-ko} is allowed.
In the proofs, we occasionally refer to the explicit scaling \eqref{eq:bal-ko2} for readability.
Similarly, for the noise threshold, with the previous choice of $\sfk_\epsilon$,  we can take
    \begin{equation}\label{eq:A-bc}
        A_\epsilon  =   \mathit{\Theta} \left( \left(\log \left(\ep^{-1} \right) \right)^{\frac {1+\lA}{2(r-1)}}  \right)
    \end{equation}

We stress the fact that the constants
$\lN, \lA >0$ can be arbitrarily small. A quantity like 
$\mathit{\Theta} \left( \sfk_\epsilon^{1 + \lA} \right)$
represents any polynomial in $\sfk_\ep$ of order larger than 1.  

The next theorem derives consistency from the specialized finite-sample bound
\autoref{thm:est-bound}, under suitable balance conditions on graph growth, sample size, and noise level.

\begin{remark}\label{rem:l-inf}
    Super-spreader hubs and $\ell_\infty$ norm. We point out that in the case $r = \infty$ condition can be replaced by the weaker 
    \[\exp\bigg(
    (\constq \vee 1)\max_{i\in[\sfp_\ep]}\sfk^\ep_i\bigg) = 
     \mathcal O \left (\epsilon^{-\lS} \right).
    \]
    In particular the condition above does not restrict the out-degree while still implying a logarithmic growth in the in-degree
    \beas
     \sfk_\epsilon= \mathcal O( \log \sfp_\epsilon).
    \eeas
    This means that the $\ell_\infty$ is a suitable choice for dealing with graphs that exhibit super-spreader hubs, that is nodes with large out-degree. 
\end{remark}

\begin{theorem}\label{thm:cons}
Suppose that Conditions \ref{as:drift-lin}, \ref{as:diff} and \ref{as:ident}, \ref{as:mixing} are fulfilled.
Then under the balance conditions in Assumption \ref{as:balance}, for all $\nu >  0$,
\bea\label{202508220815}
P\big[\rho_\ep(\wh{\theta}_\ep^r,\theta_\ep^*)\geq\nu \big]
&\to&
0
\eea
as $\ep\to0$.
\end{theorem}

\noindent
In \autoref{sec:proof-cons}, we prove \autoref{thm:cons} by showing convergence to zero of each term on the right-hand side of 
the inequality (\ref{202508162002}) under the given set of balance conditions.

As a consequence of \autoref{thm:cons} we are able to derive rates of convergence for the estimator. 
In fact, we see in \autoref{thm:rates} that faster rates can
be obtained under milder growth conditions. 
First, we need to introduce the following modified set of balance conditions. 
\begin{enumerate}[label=(B\arabic*), resume]

\item \label{as:balance2}
\textit{(Conditions for improved rates)}. 
Given positive constants $c_1, \alpha,  \gamma, \lS$ and $\beta \geq 0$ that satisfy
\begin{enumerate}[label=(\roman*)]
    \item 
        \begin{gather}
            c_1 + \beta < 2   \label{eq:bc4}\\
            c_1 > 3 (\beta + 2\gamma) \label{eq:bc5} \\
            \lS + \gamma + \frac{\constv} 2 < \alpha \wedge 1 \label{eq:bc6},
        \end{gather}
    assume the following balance conditions:
    \item (Dimension scaling)  \[\sfp_\epsilon = \mathcal O\left( \epsilon^{-\lG} \right) \]  for some $\lG>0$; 
    \item (Graph sparsity scaling) 
    \bea \label{eq:bal-ko3}
    \exp\{(C_{3,r}\vee 1) \sfk^{r-1}_\epsilon \mathsf{o}_\epsilon \} = \mathcal O \left ( \epsilon^{-\lS} \right). \quad 
    \eea

    \item (Sample size growth)
    For some $\lN > 0$
    \bea \label{eq:bal-n3}
    {\sqrt n_\ep} =  \mathit{\Theta} \left( \ep^{-2\gamma - \beta - \lN} \right) 
    \eea 

    \item (Noise threshold)
    For some $\lA > 0$,
    \[ 
    A_\epsilon =  \mathit{\Theta} \left( \sfk_\epsilon^{1 + \lA} \right), \]
    
\end{enumerate}
\end{enumerate}

\begin{remark}
    Note that under Assumption \ref{as:balance2} a faster growth of the sample size is required compared to \ref{as:balance}. 
    Moreover a tighter constraint on $\lS$ -- controlling the graph sparsity -- appears in Assumption \ref{as:balance2}. This is due to the presence of the additional parameter $\gamma$ which controls the convergence rate of the estimator. 
    These differences reflect a different rate of convergence of the estimator, which we label \emph{slow} and \emph{improved} rates in the next theorem. Intuitively, under milder dimension growth, and a larger order of the  sample size,  faster convergence of the estimator can be achieved. 
\end{remark}



\begin{theorem}\label{thm:rates}
Suppose that Conditions \ref{as:drift-lin}, \ref{as:diff}, \ref{as:ident}, and \ref{as:mixing}, are fulfilled.
Then the following convergence rates hold in the metric $\rho_\ep$.

\begin{enumerate}[label=(\roman*)]
    \item (Slow rate)
    Under Assumption \ref{as:balance},
    \[
    \log(1/\ep)\,\rho_\ep(\wh{\theta}_\ep^r,\theta_\ep^*) \to 0
    \qquad\text{in probability as }\ep\to0.
    \]

    \item (Improved rate)
    Under Assumption \ref{as:balance2},
    \[
    \ep^{-\gamma}\,\rho_\ep(\wh{\theta}_\ep^r,\theta_\ep^*) \to 0
    \qquad\text{in probability as }\ep\to0.
    \]
\end{enumerate}
\end{theorem}
The proof follows the same ideas as 
\autoref{thm:cons} and is postponed to \autoref{sec:proof-rates}.

\section{Adaptive Lasso estimation}\label{sec:lasso}

In this section, we work with the full directed candidate graph, that is, we associate one parameter $\theta_\ep^{ij}$ with each ordered pair $(i,j)\in[\sfp_\ep]^2$.
The true graph is then identified with the support of the true parameter. 
Specifically,  we assume that the data are generated 
by the model \eqref{eq:interactions-model}
\bea 
S^{\ep,i}(t,\bbX^{\ep,i}_t,\theta_\ep^*)
&=& 
\sum_{j\in N_i^\ep}
\theta^{*, ij}_\ep \psi^{ij}\big(X^{\ep,i}_t,X^{\ep,j}_t\big)
+\ep^\alpha\sum_{j\in\checkN_i^\ep}V^{ij}_\ep(t,X^{\ep,i}_t,X^{\ep, j}_t\big),
\quad i\in[\sfp_\ep]
\eea 
and we suppose  that the neighborhoods $ N_i^\ep,i\in[\sfp_\ep] $ in the true graph 
$G_\epsilon = (V_\epsilon, E_\epsilon)$ are unknown. 
Our aim is to use an adaptive Lasso procedure to recover the graph structure 
and estimate the interaction parameters. 
In order to do so we adopt the following parametric model
\bea 
S^{\ep,i}(t,\bbX^{\ep,i}_t,\theta_\ep)
&=& 
\sum_{j = 1}^{\sfp_\ep}
\theta^{ij}_\ep \psi^{ij}\big(X^{\ep,i}_t,X^{\ep,j}_t\big)
+\ep^\alpha\sum_{j\in\checkN_i^\ep}V^{ij}_\ep(t,X^{\ep,i}_t,X^{\ep, j}_t\big),
\quad i\in[\sfp_\ep],
\eea
that is $\theta^{ij}_\ep = 0 $ if and only if $(i,j) \notin E_{\epsilon}$. 
In this section, the model parameters are  $(\theta_\ep^{ij})_{i,j\in[\sfp_\ep]}$ and we denote by $\theta_\ep$ the vectorized form
\[
\theta_\ep = \mathsf{vec}(\theta_\ep^{ij},\, i,j\in[\sfp_\ep]) \in \mathbb{R}^{\pi_\ep}, \quad 
\pi_\ep = \sfp_\epsilon^2
\]
and we equip the parameter space with the metric $\rho_\ep$,  coinciding with the Euclidean norm on the vectorized parameter normalized by $\sqrt{\sfp_\ep}$, that is,
\[
\rho_\ep(\theta_\ep,\theta_\ep')
=
\frac{1}{\sqrt{\sfp_\ep}}|\!|\theta_\ep-\theta_\ep'|\!|, \qquad 
\theta_\ep,\theta_\ep' \in \Theta_\epsilon.
\]

In the following denote by $\tilde \theta_\epsilon^r$ the minimum $L^r$-distance estimator, and omit the superscript $r$ for ease of reading. 
Define an adaptive Lasso estimator

\begin{equation}\label{eq:lasso-est}
    \thL \in \arg \min_{\theta_\epsilon} |\!| \theta_\epsilon - \tilde \theta_\epsilon |\!|^2 + \lambda_\epsilon |\!| \theta_\epsilon |\!|_{1, w_\epsilon}
\end{equation}
where 
\[
    |\!| \theta_\epsilon |\!|_{1, w_\epsilon} = \sum_{i,j=1}^{\sfp_\ep} w_{\epsilon}^{ij}  |\theta^{ij}_\ep |
\]
and $w_\epsilon = \mathsf{vec}(w_{\epsilon}^{ij}; i,j \in [\sfp_\epsilon]) \in \bbR^{\pi_\epsilon}$ is a data-dependent weight vector. 


\begin{remark}
    The adaptive lasso estimator \eqref{eq:lasso-est} is a special case 
    of the adaptive $\ell_q$-penalized estimator considered in 
    \cite{suzuki2020penalized}, namely
    \begin{gather*}
        \hat \theta_\ep^{(q)} \in \arg\min_\theta \mathcal L^{(q)}_\ep (\theta)
        \\
      \mathcal L_\ep^{(q)}(\theta) = \langle\widehat{\mathcal I}_\ep, (\theta-\tilde{\theta}_\ep)^{\otimes 2}\rangle
        + \lambda_\epsilon |\!| \theta |\!|_{q,  w_\ep}^q
    \end{gather*}
    where $\langle A, B\rangle$ denotes the Frobenius inner product between matrices $A, B$, 
    $\widehat{\mathcal I}_\epsilon$ denotes a data-dependent information matrix and $|\!|\theta |\!|_{q, w}$ denotes the weighted $\ell_q$ metric 
    \[
     \left(|\!|\theta |\!|_{q,  w}\right) ^q = \sum_{j=1}^{\pi_\epsilon} w_{\epsilon}^j  |\theta^j|^q
    \]
    for a parameter and a weight vector $\theta, w \in \bbR^{\pi_\epsilon}$. 
    Namely, \eqref{eq:lasso-est} corresponds to the case $\widehat{\mathcal I}_\ep = I_{\pi_\ep}$ -- the $\pi_\ep\times \pi_\ep$ identity matrix; in \cite{suzuki2020penalized} this is also called a P-O estimator. 
    An extension to multiple penalties, in the general case where multiple rates of convergence are considered, has been studied in \cite{de2021regularized}, and practical implementations of these estimators are presented in \cite{de2024pathwise}. 
\end{remark}

Once $\hat \theta^L_\epsilon$ has been obtained, we can estimate the adjacency matrix of the graph as

\[
\hat {\mathcal A}_{\epsilon}^{ij}  = 
\begin{cases}
		\mathds 1(\hat \theta^{L,ij}_\epsilon \neq 0) &  \,i \neq j \\
		0 & i=j
	\end{cases}
	\qquad 1\leq i,j \leq \sfp_\epsilon.
\]
Define $\hat G_\epsilon = (V_\epsilon, \hat E_\epsilon)$ as the corresponding estimated graph.

We impose the following conditions linking the penalty level $\lambda_\ep$, the adaptive weights, and the convergence rate $r_\ep$.

\begin{enumerate}[label=(C\arabic*)]
    \item \label{as:lasso-active}
    
    Let
    \[
    \bar w_\ep = \max_{(i,j)\in E_\ep} w_\ep^{ij}.
    \]
    We assume that
    \[
    \lambda_\ep \frac{r_\ep^{-1}}{\sqrt{\sfp_\ep}}\,\bar w_\ep \|\theta_\ep^*\|_0 = O_p(1).
    \]

    \item \label{as:lasso-inactive}
    Let
    \[
    \check w_\ep = \min_{(i,j)\notin E_\ep} w_\ep^{ij}.
    \]
    We assume that
    \[
    \frac{r_\ep^{-1}\lambda_\ep \check w_\ep}{\sqrt{\sfp_\ep}} \to \infty.
    \]
\end{enumerate}


\begin{remark}
A standard choice of adaptive weights in the adaptive-Lasso literature
(\cite{zou2006adaptive,de2021regularized,suzuki2020penalized})
is
\[
w_\ep^{ij}=|\tilde\theta_\ep^{ij}|^{-\delta},
\qquad i,j\in[\sfp_\ep],
\]
for some $\delta>0$.
Since the support size satisfies $\|\theta_\ep^*\|_0 \le \sfp_\ep \sfk_\ep$,
a sufficient condition for Assumption \ref{as:lasso-active} is
\[
\lambda_\ep
=
O\!\left(
\frac{r_\ep}{\sqrt{\sfp_\ep}\,\sfk_\ep}
\right).
\]
On the other hand, 
a sufficient condition for Assumption \ref{as:lasso-inactive} is
\[
\lambda_\ep r_\ep^{-1-\delta}\,(\sqrt{\sfp_\ep})^{-1-\delta}\to\infty.
\]

\end{remark}

In order to guarantee support recovery, we assume the following condition on the true parameter vector, which is known in the Lasso literature as the \textit{beta-min} condition (see e.g. \cite{bunea2008honest}). We adapt it here to our setting, with a diverging dimension depending on $\epsilon$. 

\begin{enumerate}[label=(C\arabic*), resume]    \item \label{as:lasso-betamin}
    \textit{(Beta-min condition).}
    Assume that 
    \[
    \theta_{\min}
    \;:=\;
    \inf_{\epsilon >0} \min_{(i,j)\in E_\ep} |\theta_{\ep}^{*,ij}| > 0 
    \]
\end{enumerate}

The following Theorem guarantees consistency and selection consistency of the adaptive lasso estimator $\hat \theta^{L}_\epsilon$.

\begin{theorem}\label{thm:lasso}
    Suppose that Conditions \ref{as:drift-lin}, \ref{as:diff}, \ref{as:ident},  \ref{as:mixing} and \ref{as:lasso-active} are fulfilled.
    
    \begin{enumerate}[label=(\roman*)]
        \item (Consistency)
        Either assume  \ref{as:balance}  and set  $r_\epsilon^{-1} = \log(\epsilon^{-1})$, or assume \ref{as:balance2} and set 
        $r_\epsilon^{-1} = \epsilon^{-\gamma}$.
       
        Then the adaptive lasso estimator $\thL$
        defined in \eqref{eq:lasso-est}
        satisfies 
        \[
          \rho_\epsilon(\thL, \theta^{*}_\epsilon ) =
         O_p(r_\epsilon).
        \]
        
        \item (Support recovery)
        Assume  \ref{as:balance2} and set 
        $r_\epsilon^{-1} = \epsilon^{-\gamma}$, with $\lG < 2 \gamma \wedge \lN / 2$, and further assume \ref{as:lasso-inactive}  and \ref{as:lasso-betamin}. Then
        \[
            P(\hat E_\epsilon = E_\epsilon) \to 1
        \]
        as $\epsilon \to 0$.
    \end{enumerate}

\end{theorem}
\noindent
The proof is given in \autoref{sec:proof-lasso}. 
We follow similar steps as in \cite[Theorem 1, Theorem 2]{suzuki2020penalized}. The main novelty is to take into account the diverging dimensionality of the parameter space. 
\begin{remark}
    The additional requirement $\lG < 2 \gamma \wedge \lN / 2$ relates the growth of the dimension to the rate $r_\ep$, namely it entails $\frac{r_\epsilon^{-1}}{\sqrt{\sfp_\epsilon}} \to \infty$. 
    This, combined with \ref{as:lasso-betamin}, allows us to achieve support recovery. 
\end{remark}

\section{Proofs}\label{sec:proofs}

\subsection{Proof of \autoref{thm:err-bound}}
\label{sec:proof-general-err}

For $r\in[1,\infty)$, 
\bea\label{202508070652}&&
P\big[\rho_\ep(\wh{\theta}_\ep^r,\theta_\ep^*)\geq\nu,\>T(\ep,A)=T\mid\calf_0\big]
\nn\\&\leq&
P\bigg[\inf_{\theta_\ep\in U_\ep(\nu)^c}\calh^r_\ep(\theta_\ep)\leq\calh^r_\ep(\theta_\ep^*),\>T(\ep,A)=T\big|\calf_0\bigg]
\nn\\&\leq&
P\bigg[
\sup_{t\in[0,T]}\bigg\{(\sfd_\ep)^{-1}\sum_{i\in[\sfd_\ep]}\big|X^{\ep,i}_t-z^{\ep,i}_t(\theta_\ep^*)\big|^r\bigg\}^{1/r}
\nn\\&&\hspace{20pt}
\geq 2^{-1}\inf_{\theta_\ep\in U_\ep(\nu)^c}
\sup_{t\in[0,T]}\bigg\{(\sfd_\ep)^{-1}\sum_{i\in[\sfd_\ep]}\big|z^{\ep,i}_t(\theta_\ep)-z^{\ep,i}_t(\theta_\ep^*)\big|^r\bigg\}^{1/r},\>T(\ep,A)=T\big|\calf_0\bigg]
\nn\\&\leq&
P\big[\wt{m}_r^\ep(T(\ep,A))^{1/r}\geq2^{-1}g^\ep_r(\nu)^{1/r}\big|\calf_0\big]
\nn\\&\leq&
P\big[m_r^\ep(T(\ep,A))^{1/r}\geq2^{-1}g^\ep_r(\nu)^{1/r}-\Delta_r^\ep\big|\calf_0\big]
\nn\\&\leq^{(\ref{202508070505})}&
P\bigg[2^{1-1/r}R^\ep_r(T(\ep,A))^{1/r}\exp\bigg(2^{r-1}r^{-1}{\consta}^rT^r(\max_{i\in[\sfd_\ep]}\sfk^\ep_i)^{r-1}{\sf o}_\ep\bigg)\geq2^{-1}g^\ep_r(\nu)^{1/r}-\Delta_r^\ep\big|\calf_0\bigg]
\nn\\&\leq&
P\bigg[R^\ep_\infty(T(\ep,A))\geq2^{-2+1/r}(g^\ep_r(\nu)^{1/r}-2\Delta_r^\ep)\exp\bigg(-2^{r-1}r^{-1}{\consta}^rT^r(\max_{i\in[\sfd_\ep]}\sfk^\ep_i)^{r-1}{\sf o}_\ep\bigg)\big| \calf_0^\ep\bigg]
\nn\\&&
\eea
since $R^\ep_r(t)^{1/r}\leq R^\ep_\infty(t)$. 
Then, from (\ref{202508070652}), (\ref{202508070650}) and (\ref{202508070607}), we conclude that 
\bea\label{202508070654}
P\big[\rho_\ep(\wh{\theta}_\ep,\theta_\ep^*)\geq\nu,\>T(\ep,A)=T\big|\calf_0\big]
&\leq&
\big\{2\constc\sfd_\ep \exp\big(-\Xi_r^\ep(\nu,A)^2\big)\big\}\wedge1. 
\eea
Here we used the estimate
\beas 
1\wedge(\constc\sfd_\ep y^{-1}e^{-y^2})
\yleq
(\constc\sfd_\ep) \wedge(\constc\sfd_\ep y^{-1}e^{-y^2})
\yleq
2\constc\sfd_\ep e^{-y^2}\quad(y>0)
\eeas
by
$C_0>1$, $\sfd_\ep\geq1$, and 
$1\wedge\{y^{-1}e^{-y^2}\}\leq2e^{-y^2}$ for all $y>0$.

Similarly, in Case $r=\infty$, we obtain
\bea\label{202508070740}&&
P\big[\rho_\ep(\wh{\theta}_\ep^\infty,\theta_\ep^*)\geq\nu,\>T(\ep,A)=T\big|\calf_0\big]
\nn\\&\leq&
P\big[\wt{m}_\infty^\ep(T)\geq2^{-1}g^\ep_\infty(\nu),\>T(\ep,A)=T\big|\calf_0\big]
\nn\\&\leq&
P\big[m_\infty^\ep(T)\geq2^{-1}g^\ep_\infty(\nu)-\Delta_\infty^\ep,\>T(\ep,A)=T\big|\calf_0\big]
\nn\\&\leq^{(\ref{202508070516})}&
P\bigg[R^\ep_\infty(T(\ep,A))\geq2^{-1}(g^\ep_\infty(\nu)-2\Delta^\ep_\infty)\exp\bigg(-T\consta\max_{i\in[\sfd_\ep]}\sfk^\ep_i\bigg)\big|\calf_0\bigg]
\nn\\&\leq&
\big\{2\constc\sfd_\ep\exp\big(-\Xi_\infty^\ep(\nu,A)^2\big)\big\}\wedge1. 
\eea

Now Inequality (\ref{202508140415}) follows from (\ref{202508070740}) and (\ref{202508070654}).

\subsection{Proof of \autoref{thm:est-bound}}\label{sec:proof-est-bound}

The 
 proof is based on an application of Theorem \ref{thm:err-bound}. Our goal is to show that, by choosing 
 $\xi = \xi^\epsilon_r$, and under the conditions in 
 \eqref{202508161946}, the following inclusion holds:
 \[
  \Omega^\ep_0(\nu) \subset (\Xi_r^\ep(\nu,A)
 \geq   \xi_r^\epsilon ) .
 \]
\noindent
For ease of reading we write $A$, $\nu$ and omit the dependence on $\epsilon$.

Recall that  
\beas
\varphi_\ep
&=& 
\sup_{\theta_\ep\in\Theta_\ep}\big|Q^\ep(\theta_\ep)-E\big[Q^\ep(\theta_\ep)\big]\big|.
\eeas

Moreover, under \ref{as:ident}, 
for $G^\ep(\theta_\ep,\constd)$
defined in \eqref{eq:i-def}
we have
\bea\label{202508210610}
E\big[G^\ep(\theta_\ep,\constd)\big]
&=&
3^{-1}T^{-1}\ep^\constd
(\sfd_\ep)^{-1}\sum_{i\in[\sfd_\ep]}
E\bigg[ E_{\calb^\ep_i}\bigg[
\big|\sum_{j\in\checkN_i^\ep}(\theta^{ij}-\theta^{ij*}) \psi^{ij}\big(x^{\ep,i}_0,x^{\ep,j}_0\big)\big|^2\bigg]\bigg]
\nn\\&\geq&
3^{-1}T^{-1}\ep^\constd
(\sfd_\ep)^{-1}\sum_{i\in[\sfd_\ep]}
E\bigg[ E_{\calb^\ep_i}\bigg[
\big|\sum_{j\in\checkN_i^\ep}(\theta^{ij}-\theta^{ij*}) \wt{\psi}^{ij}\big|^2\bigg]\bigg]
\nn\\&\geq&
3^{-1}\constk T^{-1}\ep^{\constv}
(\sfd_\ep)^{-1}\sum_{i\in[\sfd_\ep]}
\sum_{j\in\checkN_i^\ep}|\theta^{ij}-\theta^{ij*}|^2%
\nn\\&=&
3^{-1}\constk T^{-1}\ep^{\constv}
(n \sfp_\ep)^{-1} n \sum_{i\in[\sfp_\ep]}
\sum_{j\in\checkN_i^\ep}|\theta^{ij}-\theta^{ij*}|^2 
\nn\\&=&
3^{-1}\constk T^{-1}\ep^\constv\rho_\ep(\theta_\ep,\theta_\ep^*)^2. 
\eea

Let $r\in[2,\infty)$ and $\constd>0$. 
Now, on the event 
\beas
\Omega^\ep_0(\nu)=\big\{\varphi_\ep\leq2^{-1}\inf_{\theta_\ep\in U_\ep(\nu)^c}E\big[Q^\ep(\theta_\ep)\big]\big\}, 
\eeas
by Lemma \ref{lem:ident-gap-1} and \eqref{202508210610-2}, 
\beas
g_\infty^\ep(\nu)^2
&\geq&
g_r^\ep(\nu)^{2/r}
\ygeq
g_2^\ep(\nu)
\ygeq
\ep^\constd\inf_{\theta_\ep\in U_\ep(\nu)^c}Q^\ep(\theta_\ep)-I^\ep(\constd)
\nn\\&\geq&
\ep^\constd\big(\inf_{\theta_\ep\in U_\ep(\nu)^c}E[Q^\ep(\theta_\ep)]-\varphi_\ep\big)-I^\ep(\constd)
\nn\\&\geq&
2^{-1}\ep^\constd\inf_{\theta_\ep\in U_\ep(\nu)^c}E[Q^\ep(\theta_\ep)]-I^\ep(\constd)
\nn\\&=&
2^{-1}\inf_{\theta_\ep\in U_\ep(\nu)^c}E[G^\ep(\theta_\ep,\constd)]-I^\ep(\constd)
\nn\\&\geq&
2\constr^2\ep^\constv\nu^2-I^\ep(\constd)
\eeas
for $\ep\in(0,T^{3/\constd}\wedge1]$ and $\nu>0$, 
where $\constr=(2{\sqrt{3}})^{-1}\constk^{1/2} T^{-1/2}$. 
Then, 
\bea\label{202508161945}
g_\infty^\ep(\nu)-2\Delta^\ep_r
\ygeq
g_r^\ep(\nu)^{1/r}-2\Delta^\ep_r
\ygeq
\constr\ep^{(\constv)/2}\nu-2\Delta^\ep_r
\ygeq
2^{-1}\constr\ep^{(\constv)/2}\nu
\eea
on the event $\Omega^\ep_0(\nu)$, if 
\bea\label{202508161946}
\constr^2\ep^\constv\nu^2\ygeq I^\ep(\constd)\quad\text{and}\quad
\constr\ep^{(\constv)/2}\nu\geq4\Delta^\ep_r
\eea
since
\beas
g_\infty^\ep(\nu)^2
\ygeq
g_r^\ep(\nu)^{2/r}
\ygeq
g_2^\ep(\nu)
\ygeq
\constr^2\ep^\constv\nu^2. 
\eeas
Under \ref{as:drift-lin}, Condition \ref{as:lip} is satisfied with 
\bea\label{202508220530}
\consta &=&\consti(\constj\vee1). 
\eea
For $r\in[1,\infty)$, it follows from (\ref{202508070650}), 
\bea\label{202508161947}
\Xi_r^\ep(\nu,A)
&\geq&
\constr\constn A^{-1}\ep^{-1+(\constv)/2}\nu\exp\bigg(-\consto(\max_{i\in[\sfp_\ep]}\sfk^\ep_i)^{r-1}{\sf o}_\ep\bigg) =: \xi_r^\epsilon
\eea
on the event $\Omega^\ep_0(\nu)$, if the inequalities in (\ref{202508161946}) hold. 
For $r=\infty$, we have from (\ref{202508070739}) that 
\bea
\Xi_\infty^\ep(\nu,A)
&\geq& 
2^{-1}\constr\constp A^{-1}\ep^{-1+(\constv)/2}\nu\exp\bigg(-\constq\max_{i\in[\sfp_\ep]}\sfk^\ep_i\bigg) =: \xi_\infty^\epsilon. 
\eea

\subsection{Proof of \autoref{thm:cons}}\label{sec:proof-cons}

We proceed by analyzing each term on the right-hand side 
of \autoref{thm:est-bound}. 

\paragraph{Convergence of $P\big[T(\ep,A_\ep)<T\big]$.}

%
Recall from Lemma \ref{lem:diff-bound} that 
\beas 
P\big[T(\ep,A_\ep)<T\big]
&\leq&
\sfd_\ep\times
2\pi^{-1/2}\frac{\sqrt{2T A_\ep^2\ep^2}}{x}\exp\bigg(-\frac{x^2}{2TA_\ep^2\ep^2}\bigg)
\eeas
whenever 
\bea \label{eq:a-bound2}
A_\ep>2\constz\bigg[1
+(\max_{i\in[\sfd_\ep]}\sfk^\ep_i)\sup_{i\in[\sfd_\ep]}\sup_{s\in[0,T]}\big|x^{\ep,i}_s\big|\bigg], 
\eea
where
\beas 
x &=& 
2^{-1}\constz^{-1}A_\ep(\max_{i\in[\sfd_\ep]}\sfk^\ep_i)^{-1}\exp\bigg(-\consta T\max_{i\in[\sfd_\ep]}\sfk^\ep_i\bigg) .
\eeas
Now, recall that $\consto = 2^{r-1} r^{-1} C_1^r T^r$. Notice that, for $\sfk_\ep \geq 1,  {\sf o}_\ep \geq 1, r>2 $, 
\[
C_1 T \sfk_\ep \leq  (2^{r-1} r^{-1} C_1^r T^r \vee 1)  \sfk_\ep^{r-1}  {\sf o}_\ep.
\]
Therefore, the balance condition \eqref{eq:bal-ko} entails
\footnote{We write $f(\epsilon) \gtrsim g(\epsilon)$ if there exists $C>0$ such that $f(\epsilon) \geq  C g(\epsilon)$. With some abuse of notation we use the same symbol when the inequality holds up to logarithmic factors of the form $\log^c( \epsilon^{-1}), \, c>0$ on the right-hand side, when they are inessential for the asymptotic behavior.}
\bea \label{eq:exp1-bound}
\exp\left (-C_1 T \sfk_\ep \right ) \geq \exp \left ( -(\consto \vee 1)\sfk_\ep^{r-1}  {\sf o}_\ep   \right  )
 \gtrsim  \ep^{\lS}
\eea
and, with $x$ defined above, 
\[
-\frac{x^2}{2TA_\ep^2\ep^2} \lesssim -\frac{1}{8 T \constz} \frac{1}{\ep^2 \sfk_\ep^2} \ep^{2\lS} \lesssim -\left(\frac{1}{\ep}\right)^{2(1 - \lS)}
\]
where the last inequality neglects logarithmic factors. 
Recall that $\sfd_\ep = n_\ep \sfp_\ep$, that 
$n_\ep = \mathit{\Theta}(\epsilon^{-2(\beta + \lN)})$
according to condition  \eqref{eq:bal-n}, 
and that $\sfk_\ep$ has logarithmic growth, 
for instance as in 
\eqref{eq:bal-ko2}. Then, by using again \eqref{eq:exp1-bound},  we have
\begin{align} \label{eq:te-bound}
P\big[T(\ep,A_\ep)<T\big] & 
\lesssim
\sqrt T 
\sfd_\ep \,  \ep \, \sfk_\ep \exp \left( T C_1 \sfk_\ep \right) 
\exp \left( -\ep^{- 2(1 -  \lS)} \right)
\notag \\
& \lesssim
\ep^{1 - \lG - \lS - 2(\beta + \lN)} \exp \left\{-\left(\frac{1}{\ep}\right)^{2(1 - \lS)}\right\}
\end{align}
 up to logarithmic terms. Exponential  convergence of the last term to zero is guaranteed if $\lS < 1$, which is ensured by \eqref{eq:bc3},
 provided that  \eqref{eq:a-bound2} holds. 
This is in fact the case  under a mild growth condition for $A_\ep$: for instance it is enough to assume \eqref{eq:A-bc}. 

\paragraph{Convergence of $\sfd_\ep e^{-(\xi^\ep_r)^2}$.}

We show that $\xi^\ep_r \to \infty$ quasi-polynomially fast in $\ep$. Then the conclusion follows due to the polynomial growth (up to logarithmic factors) of $\sfd_\ep = n_\ep \sfp_\ep$ in Assumption \ref{as:balance}. From \eqref{eq:xir-def}

\begin{align*}
    \xi^\ep_r & = 
\constr\constn A_\ep^{-1}\ep^{-1+(\constv)/2}\nu \exp\bigg(-\consto(\max_{i\in[\sfp_\ep]}\sfk^\ep_i)^{r-1}{\sf o}_\ep\bigg)
\\
& \gtrsim \left( \log \frac 1 \ep \right)^{-\frac{1 + \lA}{2(r-1)}} \ep^{-1+(\constv)/2} \epsilon^{\lS},
\end{align*}
which diverges under the balance conditions in Assumption \ref{as:balance}.

\paragraph{Convergence of $\mathds 1_{\big\{
\constr\ep^{(\constv)/2}\nu_\ep<(I^\ep(\constd)^{1/2})\vee(4\Delta^\ep_r)\big\}}$.}

We start by analyzing $I^\ep(\constd)$, which is defined in \eqref{eq:i-def} as
\[
I^\ep(\constd)
\yeq
\constl
(\max_{i}\sfk^\ep_i)^3\ep^{\frac{4}{3}\constd}
\exp\bigg(\constm\ep^{\constd/3} \sfk_\ep\bigg).
\]
From condition \eqref{eq:bal-ko2}, $\ep^{\constd/3} \sfk_\ep \to 0$, and therefore for $\ep$ small enough, 
$
I^\ep(\constd) \lesssim \ep^{\frac{4}{3} \constd}
$,
neglecting logarithmic factors. This means that the asymptotic behavior of the inequality
\begin{equation}\label{eq:step3-1}
    \constr^2\ep^{\constv}\nu^2 \geq  I^\ep(\constd)
\end{equation}
is  determined by $\ep^{4\constd/3 } / \ep^{\constv}$, which tends to zero if \eqref{eq:bc2}, namely $c_1 >
3 \beta$, holds true. 
Recall that, by Lemma \ref{lem:delta},
\bea\label{eq:delta-inf-bound-app}
    \Delta^\ep_r \leq \Delta^\ep_\infty
     &\leq&
    \consti\ep^\alpha\sfk_\ep T
    \bigg\{1+ \constj
    +
    T\consti \sfk_\ep\bigg\}
    \exp\big(T(\consti+\consti\constj) \sfk_\ep\big). 
    \eea

Then, by \eqref{eq:delta-inf-bound-app}, and by the conditions in Assumption \ref{as:balance}
\begin{align*}
    \Delta^\ep_\infty &\leq 
    \consti\ep^\alpha\sfk_\ep T
    \bigg\{1+ \constj
    +
    T\consti \sfk_\ep\bigg\}
    \exp\big(T(\consti+\consti\constj) \sfk_\ep\big) \\
    & \leq 
    \consti\ep^\alpha T \sfk_\ep^2
    \bigg\{1+ \constj + T\consti \bigg\}
    \exp\big(\consto \sfk_\ep^{r-1} {\sf o}_\ep \big) 
    \\
    & \lesssim
    \ep^{\alpha - \lS}
\end{align*}
where we neglect logarithmic factors in the domination. Finally, the 
inequality $\constr\ep^{(\constv)/2}\nu \geq4\Delta^\ep_\infty$ holds eventually for small $\ep$  
if $\ep^{\alpha - \lS} \cdot  \ep^{-\frac{\constv}{2}} \to 0$, which holds true under 
\eqref{eq:bc3}, namely if $\alpha - \lS > (\constv)/2$.

\paragraph{Convergence of $P\big[(\Omega^\ep_0(\nu_\ep))^c\big]$.}
The convergence $P\big[(\Omega^\ep_0(\nu_\ep))^c\big]\to0$ follows from 
Lemma \ref{lem:omega-bound}, which gives
\beas
P\big[(\Omega^\ep_0(\nu))^c\big]
&\leq&
{16\sqrt{2}}(\constj)^2\constu\constk^{-1}\>
\bigg\{\sum_{h=0}^{n-1}(1-n^{-1}h)\alpha^\ep(h)\bigg\}^{1/2}
\epsilon^{-\beta}
\nu^{-2}
(\max_{i\in[\sfp_\ep]}\sfk^\ep_i)^2\>
n^{-1/2}.
\eeas
Set $n = n_\ep$. Under Assumption \ref{as:mixing} 
we have that
\[
P\big[(\Omega^\ep_0(\nu))^c\big] \lesssim \frac{\sfk_\ep^2}{\epsilon^{\beta} \sqrt{n_\ep}} \to 0
\]
as $\ep \to 0$, due to \eqref{eq:bal-n}.

\begin{remark}\rm 
A large deviation estimate for $P\big[\Omega^0_\ep(\nu)^c\big]$ is possible. 
It would give a sharper bound than (\ref{202508210651}). 
\end{remark}

\subsection{Proof of \autoref{thm:lasso}}
\label{sec:proof-lasso}
 
        To prove $(i)$ we show that 
        \[
        r_\epsilon^{-1} \rho_\epsilon(\hat \theta^L_\epsilon, \theta^*_\epsilon )
        = 
        \frac{r_\epsilon^{-1} }{\sqrt{\sfp_\ep}}
            |\!|\hat \theta^L_\epsilon - \theta^*_\epsilon |\!|
            \leq  O_p(1). 
        \]
        
        Let 
        \[
        \mathcal L_\epsilon(\theta_\epsilon) = |\!| \theta_\epsilon - \tilde \theta_\epsilon |\!|^2 + \lambda_\epsilon |\!| \theta |\!|_{1, w_\epsilon}. 
        \]
        Then 
        \begin{align*}
            0 & \geq 
             \mathcal L_\epsilon(\hat \theta^L_\epsilon) - 
             \mathcal L_\epsilon(\theta^*_\epsilon)
            \\ &
            =
            |\!| \hat \theta^L_\epsilon - \theta^*_\epsilon |\!|^2 
            -2 
            (\hat \theta^L_\epsilon - \theta^*_\epsilon)^\top
            (\tilde  \theta_\epsilon - \theta^*_\epsilon)
            + \lambda_\epsilon
            \sum_{ij} w_{\epsilon}^{ij} 
                (|\hat \theta^{L, ij}_\epsilon | - | \theta^{*, ij}_\epsilon |)
            \\ & \geq 
             |\!| \hat \theta^L_\epsilon - \theta^*_\epsilon |\!|^2 
            -2 
            |\!|\hat \theta^L_\epsilon - \theta^*_\epsilon |\!|
            |\!|\tilde  \theta_\epsilon - \theta^*_\epsilon|\!|
            - \lambda_\epsilon
            \bar w_\epsilon |\!| \theta^*_\epsilon |\!|_0 \, 
            |\!| \hat \theta^L_\epsilon - \theta^*_\epsilon |\!|.
        \end{align*}
        Hence 
        \begin{align}\label{eq:theta-lasso-bound}
            \frac{r_\epsilon^{-1} }{\sqrt{\sfp_\ep}}
            |\!|\hat \theta^L_\epsilon - \theta^*_\epsilon |\!|
            \leq 
            2 r_\epsilon^{-1}
            \underbrace{
            \frac{1 }{\sqrt{\sfp_\ep}}
            |\!|\tilde  \theta_\epsilon - \theta^*_\epsilon|\!|}_
            {\rho_\epsilon(\tilde \theta_\epsilon, \theta_\epsilon^*)}
            + 
            \lambda_\epsilon
            \frac{r_\epsilon^{-1} }{\sqrt{\sfp_\ep}} \bar w_\epsilon |\!| \theta^*_\epsilon |\!|_0 = O_p(1)
        \end{align}
        due to \autoref{thm:rates} and the assumptions. 

    \textit{Proof of (ii)}. 
    \textit{Step 1.}
    We first show that 
    $(\hat E_\epsilon \subset E_\epsilon)$ 
    with probability tending to 1. 
    Since we place an edge
    whenever $\thL$ is not zero, 
    $(\thLi = 0,  \;\; \forall  (i, j) \notin E_\epsilon) \subset (\hat E_\epsilon \subset E_\epsilon)$. Then it suffices to show that
	\[
	P(\thLi \neq 0 \;\; \text{for some } (i,j) \notin E_\epsilon) \to 0.
	\]
    Fix some $ (i,j) \notin E_\epsilon$ and suppose $\thLi \neq 0 $ and $ \thL \notin \partial \Theta_\epsilon$ .
    From inspection
    of the 
    Karush–Kuhn–Tucker (KKT) 
    optimality conditions we have
    \begin{align*}
        \frac{\partial}{\partial \theta_{\epsilon, ij}}
        \mathcal{\mathcal L}_\epsilon(\theta_\epsilon) \Bigg \vert _{\theta_\epsilon=\thL} =
         2 (\thLi-\tilde{\theta}_{\epsilon, ij}) + \lambda_\epsilon w_{\epsilon, ij} \, \mathrm{sgn}(\thLi) = 0
    \end{align*}
    Then 
    \begin{equation*}
       r_\epsilon^{-1}\rho_\epsilon(\thL, \tilde{\theta}_{\epsilon}) \geq  \frac{r_\epsilon^{-1}}{\sqrt{\sfp_\epsilon}}|\thLi-\tilde{\theta}_{\epsilon, ij}| 
       = 
       \frac {r_\epsilon^{-1}}{2\sqrt{\sfp_\epsilon}} \lambda_\epsilon w_{\epsilon, ij} \geq \frac {r_\epsilon^{-1}}{2\sqrt{\sfp_\epsilon}} \lambda_\epsilon  \check w_{\epsilon}. 
    \end{equation*}
    Therefore, due to \autoref{thm:rates}, part \textit{(i)}  and the assumptions, 
    \[
    P(\thLi \neq 0 , \thL \notin \partial \Theta_\epsilon ) \leq 
    P\left (r_\epsilon^{-1} \rho_\epsilon(\thL, \tilde{\theta}_{\epsilon}) \geq 
    \frac 1{2\sqrt{\sfp_\epsilon}} \lambda_\epsilon  r_\epsilon^{-1} \check w_{\epsilon} \right).  
    \]
     The right-hand side of the above formula converges to zero, and this suffices in the fixed dimension case. In the diverging dimension case, we need a quantitative estimate for the term above, due to the fact that there is 
     a diverging number of terms of the form 
     $P(\thLi \neq 0 , \thL \notin \partial \Theta_\epsilon )$. 
First, note that 
\begin{align*}
    r_\epsilon^{-1} \rho_\epsilon(\thL, \tilde{\theta}_{\epsilon})
    & \leq 
    r_\epsilon^{-1} 
    \rho_\epsilon(\thL, \theta_{\epsilon}^*) + 
    r_\epsilon^{-1} 
    \rho_\epsilon(\tilde{\theta}_{\epsilon}, \theta_{\epsilon}^*)
    \\ & \stackrel{\eqref{eq:theta-lasso-bound}}{\leq}
    2 r_\epsilon^{-1}
            {\rho_\epsilon(\tilde \theta_\epsilon, \theta_\epsilon^*)}
            + 
            \lambda_\epsilon
            \frac{r_\epsilon^{-1} }{\sqrt{\sfp_\ep}} \bar w_\epsilon |\!| \theta^*_\epsilon |\!|_0 + 
            r_\epsilon^{-1} 
    \rho_\epsilon(\tilde{\theta}_{\epsilon}, \theta_{\epsilon}^*). 
\end{align*}
By the assumption on the rates,  we can find  $\bar \nu > 0$ such that
eventually, for all $\epsilon$ small enough, 
\[
 \lambda_\epsilon  \frac{r_\epsilon^{-1} }{6\sqrt{\sfp_\ep}} \big( \check w_{\epsilon} -  2\bar w_\epsilon |\!| \theta^*_\epsilon |\!|_0 \big) \geq \bar \nu.
\]

Therefore, by \autoref{thm:rates} and its proof, 
\begin{align}\label{eq:quant-lasso-err}
    P\left (r_\epsilon^{-1} \rho_\epsilon(\thL, \tilde{\theta}_{\epsilon}) \geq 
    \frac 1{2\sqrt{\sfp_\epsilon}} \lambda_\epsilon  r_\epsilon^{-1} \check w_{\epsilon} \right)
    & \leq 
    P\left ( 
    3 r_\epsilon^{-1} 
    \rho_\epsilon(\tilde{\theta}_{\epsilon}, \theta_{\epsilon}^*)
    \geq  \frac 1{2\sqrt{\sfp_\epsilon}} \lambda_\epsilon  r_\epsilon^{-1} \big( \check w_{\epsilon} -  2\bar w_\epsilon |\!| \theta^*_\epsilon |\!|_0 \big )
    \right) \notag \\
    & \leq 
    P\left ( 
    r_\epsilon^{-1} 
    \rho_\epsilon(\tilde{\theta}_{\epsilon}, \theta_{\epsilon}^*)
    \geq \bar \nu
    \right)
     \notag \\ &
    \stackrel{\eqref{eq:err-bound-O}}{\lesssim}
O\left(\left(\log\ep^{-1}\right)^{1/(r-1)}  \epsilon^{\lN} \right)     
\end{align}
Then we conclude that
    \begin{align*}
         P(\thLi \neq 0 \;\; \text{for some } (i,j) \notin E_\epsilon)
   &\leq 
   P( \thL \in \partial \Theta_\epsilon ) + 
   \sum_{(i, j) \notin E_\epsilon} P(\thLi \neq 0 , \thLi \notin \partial \Theta_\epsilon )
   \\
   & \lesssim
   P( \thL \in \partial \Theta_\epsilon )
   + 
   (\sfp_\epsilon^2 - \sfk_\epsilon)
   O\left(\left(\log\ep^{-1}\right)^{1/(r-1)}  \epsilon^{\lN} \right)   
   \\
   & \lesssim
   P( \thL \in \partial \Theta_\epsilon )
   + O\left(\left(\log\ep^{-1}\right)^{1/(r-1)}  \epsilon^{\lN - 2 \lG} \right)  \to 0 
    \end{align*}
  where we used the fact that, 
  by assumption, $\sfp_\epsilon =  O (\epsilon^{-\lG})$,  $\lG < \lN / 2$, 
  while  $\sfk_\epsilon = O(\log \sfp_\epsilon)$, 
  hence negligible.
This proves that no spurious edges are selected with probability tending to one.
    \textit{Step 2.}
    Next we focus on the event $(\hat E_\epsilon \supset E_\epsilon)$. 
    Denote by 
    \[
    \mathrm{sgn}(x) = \begin{cases}
        +1 & x > 0 \\
        0 & x=0 \\
        -1 & x < 0.
    \end{cases}
    \]
    It holds that
    \[
    \big(\mathrm{sgn} ( \thLi) = \mathrm{sgn} ( \theta^*_{\epsilon, ij}) \,\, \forall i,j \in E_\epsilon \big)
    \subset 
    \big(\thLi \neq 0 \,\, \forall i,j \in E_\epsilon  \big) \subset ( \hat E_\epsilon \supset E_\epsilon ).
    \]
    For any $(i, j) \in E_\epsilon$, supposing without loss of generality $\mathrm{sgn} (\theta^*_{\epsilon, ij}) = +1$, we have 
    \begin{align*}
        &
        P\big(\mathrm{sgn} ( \thLi) \neq \mathrm{sgn} ( \theta^*_{\epsilon, ij}) \big)
        = 
        P( \thLi < 0) 
        \\ &= 
        P\left( \frac{r_\epsilon^{-1}}{\sqrt \sfp_\epsilon} (\thLi -  \theta^*_{\epsilon, ij}) < 
        - \frac{r_\epsilon^{-1}}{\sqrt \sfp_\epsilon} \theta^*_{\epsilon, ij} \right)
        \\ & \leq
        P\left( r_\epsilon^{-1} \rho_\epsilon(\thL, \theta^*_{\epsilon}) > 
         \frac{r_\epsilon^{-1}}{\sqrt \sfp_\epsilon} \theta_{\min} \right) 
    \end{align*}
    and recall that 
     $ \frac{r_\epsilon^{-1}}{\sqrt \sfp_\epsilon} \gtrsim \epsilon^{- \gamma + \lG/2} \to \infty $ and then  the probability above converges to zero. In order to obtain a quantitative estimate we argue as in 
    \eqref{eq:quant-lasso-err}, and we get
    \begin{align*}
         P\left( r_\epsilon^{-1} \rho_\epsilon(\thL, \theta^*_{\epsilon}) > 
         \frac{r_\epsilon^{-1}}{\sqrt \sfp_\epsilon} \theta_{\min} \right) & 
         \leq 
          P\left( 2r_\epsilon^{-1} \rho_\epsilon(\tilde \theta_\epsilon, \theta^*_{\epsilon}) > 
          \frac{r_\epsilon^{-1}}{\sqrt \sfp_\epsilon}
         \left(  \theta_{\min} - \lambda_\epsilon \bar w_\epsilon
         |\!| \theta^*_{\epsilon}|\!|_0
         \right) \right) \\
         & \lesssim O\left(\left(\log\ep^{-1}\right)^{1/(r-1)}  \epsilon^{\lN} \right) 
    \end{align*}

    In conclusion, recalling that by Assumption \ref{as:balance2}, $|E_\epsilon| \leq \sfp_\ep \sfk_\epsilon \leq O(\sfp_\epsilon \log \sfp_\epsilon) \leq O(\log \epsilon^{-1} \epsilon^{-\lG})$, 
    \begin{align*}
         P( \hat E_\epsilon \not\supset E_\epsilon ) &\leq \sum_{(i,j) \in E_\epsilon} 
    P\big(\mathrm{sgn} ( \thLi) \neq \mathrm{sgn} ( \theta^*_{\epsilon, ij}) \big) 
    \leq 
    \sfk_\epsilon  \sfp_\epsilon
    P\left( r_\epsilon^{-1} \rho_\epsilon(\thL, \theta^*_{\epsilon}) > 
         \frac{r_\epsilon^{-1}}{\sqrt \sfp_\epsilon} \theta_{\min} \right) 
    \\
    &\lesssim
    O\left(\left(\log\ep^{-1}\right)^{1/(r-1) +1 }  \epsilon^{\lN - \lG} \right) 
    \to 0.
    \end{align*}
   
  This proves that all true edges are retained with probability tending to one.

\subsection{Proof of Lemma \ref{lem:ini-bound}}\label{sec:proof-ini-bound}
    
We have
\beas 
\big|S^{\ep,i}(t,x,\theta_\ep)-S^{\ep,i}(t,y,\theta_\ep)\big|
&\leq&
\consta(\sfk^\ep_i)^{1-\frac{1}{r}}\bigg\{\sum_{k\in\caln^\ep_i}\big|x_k-y_k\big|^r\bigg\}^{1/r}
\nn\\&=&
\consta\sfk^\ep_i\bigg\{(\sfk^\ep_i)^{-1}\sum_{k\in\caln^\ep_i}\big|x_k-y_k\big|^r\bigg\}^{1/r}
\quad(r\in[1,\infty))
\eeas
and 
\beas 
\big|S^{\ep,i}(t,x,\theta_\ep)-S^{\ep,i}(t,y,\theta_\ep)\big|
&\leq&
\consta\sfk^\ep_i\max_{k\in\caln^\ep_i}\big|x_k-y_k\big|
\eeas
for $x=(x_k)_{k\in\caln^\ep_i}$ and $y=(y_k)_{k\in\caln^\ep_i}$. 

We remark that these estimates give the same magnitude of coefficient $\sfk^\ep_i$. 

\beas 
m_r^\ep(t) 
&\leq&
2^{r-1}\sup_{s\in[0,t]}(\sfd_\ep)^{-1}\sum_{i\in[\sfd_\ep]}\bigg|
\int_0^s\bigg(S^{\ep,i}(v,\bbX^{\ep,i}_t,\theta_\ep^*)-S^{\ep,i}(v,x^{\ep,i}_t,\theta_\ep^*)\bigg)dv\bigg|^r
+2^{r-1}R^\ep_r(t)
\nn\\&\leq^{(\ref{202508070344})}&
2^{r-1}\sup_{s\in[0,t]}(\sfd_\ep)^{-1}t^{r-1}\sum_{i\in[\sfd_\ep]}
\int_0^s{\consta}^r\bigg(\sum_{k=1}^{\sfk^\ep_i}\big|(\bbX^{\ep,i}_v)_k-(x^{\ep,i}_v)_k\big|\bigg)^rdv
+2^{r-1}R^\ep_r(t)
\nn\\&\leq&
2^{r-1}\sup_{s\in[0,t]}(\sfd_\ep)^{-1}{\consta}^rt^{r-1}(\max_{i\in[\sfd_\ep]}\sfk^\ep_i)^{r-1}\sum_{i\in[\sfd_\ep]}
\int_0^s\sum_{k=1}^{\sfk^\ep_i}\big|(\bbX^{\ep,i}_v)_k-(x^{\ep,i}_v)_k\big|^rdv
+2^{r-1}R^\ep_r(t)
\nn\\&\leq&
2^{r-1}\sup_{s\in[0,t]}(\sfd_\ep)^{-1}{\consta}^rt^{r-1}(\max_{i\in[\sfd_\ep]}\sfk^\ep_i)^{r-1}
\int_0^s {\sf o}_\ep\sum_{i\in[\sfd_\ep]}\big|X^{\ep,i}_v-x^{\ep,i}_v\big|^rdv
+2^{r-1}R^\ep_r(t)
\nn\\&\leq&
2^{r-1}{\consta}^rT^{r-1}(\max_{i\in[\sfd_\ep]}\sfk^\ep_i)^{r-1}{\sf o}_\ep
\int_0^t m_r^\ep(v)dv
+2^{r-1}R^\ep_r(t)
\eeas
for $r\in[1,\infty)$. 
Therefore, 
\bea
m^\ep_r(t)
&\leq&
2^{r-1}R^\ep_r(t)\exp\bigg(2^{r-1}{\consta}^rT^{r-1}t(\max_{i\in[\sfd_\ep]}\sfk^\ep_i)^{r-1}{\sf o}_\ep\bigg)
\eea
for $t\in[0,T]$ and $r\in[1,\infty)$.

%
On the other hand,  
\beas 
m_\infty^\ep(t) 
&\leq&
\sup_{s\in[0,t]}\max_{i\in[\sfd_\ep]}\bigg|
\int_0^s\bigg(S^{\ep,i}(v,\bbX^{\ep,i}_v,\theta_\ep^*)-S^{\ep,i}(v,x^{\ep,i}_v,\theta_\ep^*)\bigg)dv\bigg|
+R^\ep_\infty(t)
\nn\\&\leq^{(\ref{202508070344})}&
\sup_{s\in[0,t]}\max_{i\in[\sfd_\ep]}
\int_0^s{\consta}\sum_{k=1}^{\sfk^\ep_i}\big|(\bbX^{\ep,i}_v)_k-(x^{\ep,i}_v)_k\big|dv
+R^\ep_\infty(t)
\nn\\&\leq&
\consta(\max_{i\in[\sfd_\ep]}\sfk^\ep_i)
\int_0^tm_\infty^\ep(v)dv
+R^\ep_\infty(t).
\eeas
Therefore, 
\bea
m_\infty^\ep(t)
&\leq&
R^\ep_\infty(t)\exp\bigg(t\consta\max_{i\in[\sfd_\ep]}\sfk^\ep_i\bigg).
\eea

\subsection{Proof of Lemma \ref{lem:stoch-rem-bound}}
\label{sec:proof-stoch-rem}
Since $R^\ep_r(t)^{1/r}\leq R^\ep_\infty(t)$ for $r\in[1,\infty)$, we will estimate $R^\ep_\infty(T)$. 
Let $\ep\in(0,1]$ and $i\in[\sfd_\ep]$. 
Let $\tau_t^{\ep,i}=\inf\{u;\int_0^u\big(\sigma^{\ep,i}(v)\big)^2dv>t\}$ and $\wt{\calf}_t=\calf_{\tau_t^{\ep,i}}$. 
By the martingale representation theorem, on some extension $(\check{\Omega}^{\ep,i},\check{\calf}^{\ep,i},\check{\bbF}^{\ep,i},\check{P}^{\ep,i})$ of 
$(\Omega,\calf,\bbF,P)$, 
there exists a Brownian motion $B^{\ep,i}=(B_t^{\ep,i})_{t\in[0,\infty)}$ with respect to $(\wt{\calf}_t^{\ep,i})_{t\in\bbR_{\geq0}}$ 
such that 
\beas 
\int_0^{t\wedge T(\ep,A)}\sigma^{\ep,i}(v)dw^i_v
&=& 
B^{\ep,i}\bigg(\int_0^{t\wedge T(\ep,A)}\big(\sigma^{\ep,i}(v)\big)^2dv\bigg)\qquad(t\in[0,T]). 
\eeas
Then, Condition \ref{as:diff} gives
\beas 
\sup_{t\in[0,T]}\bigg|\int_0^{t\wedge T(\ep,A)}\sigma^{\ep,i}(v)dw^i_v\bigg|
&\leq&
\sup_{t\in[0,TA^2\ep^2]}|B^{\ep,i}(t)|. 
\eeas
It is known that 
\beas 
\check{P}\bigg[\sup_{t\in[0,TA^2\ep^2]}|B^{\ep,i}(t)|>x\bigg]
&\leq&
2\pi^{-1/2}\frac{\sqrt{2TA^2\ep^2}}{x}\exp\bigg(-\frac{x^2}{2TA^2\ep^2}\bigg)\qquad(x>0). 
\eeas
As a result, 
\bea
P\big[R^\ep_\infty(T(\ep,A))>x\mid\calf_0\big]
&\leq&
\sum_{i\in[\sfd_\epsilon]}
P\bigg[\sup_{t\in[0,T]}\bigg|\int_0^{t\wedge T(\ep,A)}\sigma^{\ep,i}(v)dw^i_v\bigg|>x\>\big|\>\calf_0\bigg]
\nn\\&\leq&
\sfd_\ep\check{P}^{\ep,i}\bigg[\sup_{t\in[0,TA^2\ep^2]}|B^{\ep,i}(t)|>x\>\big|\>\calf_0\bigg]
\nn\\&=&
\sum_{i\in[\sfd_\epsilon]}E^{\check{P}^{\ep,i}}\bigg[\check{P}^{\ep,i}\bigg[\sup_{t\in[0,TA^2\ep^2]}|B^{\ep,i}(t)|>x\>\big|\>\check{\calf}_0\bigg]\>\big|\>\calf_0\bigg]
\nn\\&=&
\sum_{i\in[\sfd_\epsilon]}E^{\check{P}^{\ep,i}}\bigg[\check{P}^{\ep,i}\bigg[\sup_{t\in[0,TA^2\ep^2]}|B^{\ep,i}(t)|>x\bigg]\>\big|\>\calf_0\bigg]
\nn\\&\leq&
\constc\sfd_\ep\frac{\sqrt{2TA^2\ep^2}}{x}\exp\bigg(-\frac{x^2}{2TA^2\ep^2}\bigg)\qquad(x>0)
\eea
with $\constc=2\pi^{-1/2}>1$.

\subsection{Proof of Lemma \ref{lem:diff-bound}}\label{sec:proof-diff-bound}

\beas 
\sup_{i\in[\sfd_\ep]}\big|X^{\ep,i}_s-x^{\ep,i}_s\big|
&\leq&
\sup_{i\in[\sfd_\ep]}\bigg|
\int_0^s\bigg(S^{\ep,i}(v,\bbX^{\ep,i}_t,\theta_\ep^*)-S^{\ep,i}(v,x^{\ep,i}_t,\theta_\ep^*)\bigg)dv\bigg|
+\sup_{i\in[\sfd_\ep]}\bigg|\int_0^s\sigma^{\ep,i}(v)dw^i_v\bigg|
\nn\\&\leq^{(\ref{202508070344})}&
\int_0^s{\consta}\bigg(\sum_{k=1}^{\sfk^\ep_i}\big|(\bbX^{\ep,i}_v)_k-(x^{\ep,i}_v)_k\big|\bigg)dv
+\sup_{i\in[\sfd_\ep]}\bigg|\int_0^s\sigma^{\ep,i}(v)dw^i_v\bigg|
\nn\\&\leq&
(\max_{i\in[\sfd_\ep]}\sfk^\ep_i)\int_0^s{\consta}\bigg(\sup_{i\in[\sfd_\ep]}\big|X^{\ep,i}_v-x^{\ep,i}_v\big|\bigg)dv
+\sup_{i\in[\sfd_\ep]}\sup_{s\in[0,T]}\bigg|\int_0^s\sigma^{\ep,i}(v)dw^i_v\bigg|
\eeas
and hence, 
\beas 
\sup_{i\in[\sfd_\ep]}\sup_{s\in[0,T]}\big|X^{\ep,i}_s-x^{\ep,i}_s\big|
&\leq&
\bigg(\sup_{i\in[\sfd_\ep]}\sup_{s\in[0,T]}\bigg|\int_0^s\sigma^{\ep,i}(v)dw^i_v\bigg|\bigg)
\exp\bigg(\consta T\max_{i\in[\sfd_\ep]}\sfk^\ep_i\bigg). 
\eeas
Then 
\beas 
\sup_{i\in[\sfd_\ep]}
\big|\sigma^{\ep,i}(t)\big|
&\leq&
\constz\ep\bigg[1
+(\max_{i\in[\sfd_\ep]}\sfk^\ep_i)\sup_{i\in[\sfd_\ep]}\sup_{s\in[0,T]}\big|x^{\ep,i}_s\big|
\nn\\&&
+(\max_{i\in[\sfd_\ep]}\sfk^\ep_i)\bigg(\sup_{i\in[\sfd_\ep]}\sup_{s\in[0,t]}\bigg|\int_0^s\sigma^{\ep,i}(v)dw^i_v\bigg|\bigg)
\exp\bigg(\consta T\max_{i\in[\sfd_\ep]}\sfk^\ep_i\bigg)\bigg]
\eeas
%
We have 
\beas&&
P\big[T(\ep,A_\ep)<T\big]
\nn\\&\leq&
P\bigg[\max_{i\in[\sfd_\ep]}\sup_{t\in[0,T(\ep,A_\ep)]}|\sigma^{\ep,i}(t)|\geq A_\ep\ep\bigg]
\nn\\&\leq&
P\bigg[
1
+(\max_{i\in[\sfd_\ep]}\sfk^\ep_i)\sup_{i\in[\sfd_\ep]}\sup_{s\in[0,T]}\big|x^{\ep,i}_s\big|
\nn\\&&
+(\max_{i\in[\sfd_\ep]}\sfk^\ep_i)\bigg(\sup_{i\in[\sfd_\ep]}\sup_{s\in[0,T(\ep,A_\ep)]}\bigg|\int_0^s\sigma^{\ep,i}(v)dw^i_v\bigg|\bigg)
\exp\bigg(\consta T\max_{i\in[\sfd_\ep]}\sfk^\ep_i\bigg)
\geq\constz^{-1}A_\ep\bigg]
\nn\\&\leq&
P\bigg[
\sup_{i\in[\sfd_\ep]}\sup_{s\in[0,T(\ep,A_\ep)]}\bigg|\int_0^s\sigma^{\ep,i}(v)dw^i_v\bigg|
\geq2^{-1}\constz^{-1}A_\ep(\max_{i\in[\sfd_\ep]}\sfk^\ep_i)^{-1}\exp\bigg(-\consta T\max_{i\in[\sfd_\ep]}\sfk^\ep_i\bigg)\bigg]
\nn\\&\leq&
\sfd_\ep\times
2\pi^{-1/2}\frac{\sqrt{2T A_\ep^2\ep^2}}{x}\exp\bigg(-\frac{x^2}{2TA_\ep^2\ep^2}\bigg)
\eeas
whenever 
\bea \label{eq:a-cond-proof}
A_\ep>2\constz\bigg[1
+(\max_{i\in[\sfd_\ep]}\sfk^\ep_i)\sup_{i\in[\sfd_\ep]}\sup_{s\in[0,T]}\big|x^{\ep,i}_s\big|\bigg] 
\eea
holds, where
\beas 
x &=& 
2^{-1}\constz^{-1}A_\ep(\max_{i\in[\sfd_\ep]}\sfk^\ep_i)^{-1}\exp\bigg(-\consta T\max_{i\in[\sfd_\ep]}\sfk^\ep_i\bigg) .
\eeas

\subsection{Proof of Lemma \ref{lem:ident-gap-1}} \label{sec:proof-ident-1}
Comparing $\psi^{ij}\big(z^{\ep,i}_s(\theta_\ep),z^{\ep,j}_s(\theta_\ep)\big)$ and $\psi^{ij}\big(x^{\ep,i}_0,x^{\ep,j}_0\big)$ for $j\in  N^\ep_i$, we see 
\bea\label{202508110330} &&
\bigg|z^{\ep,i}_t(\theta_\ep)-z^{\ep,i}_t(\theta_\ep^*)
-t\sum_{j\in\checkN_i^\ep}(\theta^{ij}-\theta^{ij*}) \psi^{ij}\big(x^{\ep,i}_0,x^{\ep,j}_0\big)\bigg|
\nn\\&=&
\bigg|\sum_{j\in\checkN_i^\ep}
\int_0^{t}\theta^{ij} \psi^{ij}\big(z^{\ep,i}_s(\theta_\ep),z^{\ep,j}_s(\theta_\ep)\big)ds
-\sum_{j\in\checkN_i^\ep}
\int_0^{t}\theta^{ij*} \psi^{ij}\big(z^{\ep,i}_s(\theta_\ep^*),z^{\ep,j}_s(\theta_\ep^*)\big)ds
\nn\\&&
-t\sum_{j\in\checkN_i^\ep}(\theta^{ij}-\theta^{ij*}) \psi^{ij}\big(x^{\ep,i}_0,x^{\ep,j}_0\big)\bigg|
\nn\\&\leq&
\consti\constj
(\max_{i}\sfk_i^\ep)t
\max_{j\in\checkN^\ep_i\cup\{i\}}\sup_{s\in[0,t]}\big|z^{\ep,j}_s(\theta_\ep)-x^{\ep,j}_0\big| 
\nn\\& & \quad + \consti\constj
(\max_{i}\sfk_i^\ep)t \max_{j\in\checkN^\ep_i\cup\{i\}}\sup_{s\in[0,t]}\big|z^{\ep,j}_s(\theta_\ep^*)-x^{\ep,j}_0\big| 
. 
\eea

Under \ref{as:drift-lin}, 
\bea\label{202508110543}
\sup_{i,j,\ep,\omega}\big|\psi^{ij}\big(x^{\ep,i}_0,x^{\ep,j}_0\big)\big|\yleq\consti(1+\constt)\><\>\infty. 
\eea
Then
\beas
\max_{i\in[\sfd_\ep]}\big|z^{\ep,i}_t(\theta_\ep)-x^{\ep,i}_0\big|
&\leq&
\max_{i\in[\sfd_\ep]}\bigg|\int_0^t \sum_{j\in\checkN^\ep_i}\theta^{ij}\psi^{i,j}(z^{\ep,i}_s(\theta_\ep),z^{\ep,j}_s(\theta_\ep))ds\bigg|
\nn\\&\leq&
\consti\constj\int_0^t
\max_{i\in[\sfd_\ep]}
\sum_{j\in N^\ep_i}\big|z^{\ep,j}_s(\theta_\ep)-x^{\ep,j}_0\big|ds
+\constj t\max_{i\in[\sfd_\ep]}\sum_{j\in N^\ep_i}|\psi^{i,j}(x^{\ep,i}_0,x^{\ep,j}_0)|
\eeas
and hence 
\beas
\sup_{s\in[0,u]}\max_{i\in[\sfd_\ep]}\big|z^{\ep,i}_s(\theta_\ep)-x^{\ep,i}_0\big|
&\leq&
\consti\constj(\max_{i\in[\sfd_\ep]}\sfk^\ep_i)\int_0^u\sup_{r\in[0,s]}\max_{i\in[\sfd_\ep]}\big|z^{\ep,j}_r(\theta_\ep)-x^{\ep,j}_0\big| ds
\nn\\&&
+\consti(1+\constt)\constj\> t(\max_{i\in[\sfd_\ep]}\sfk^\ep_i)
\eeas
for $u\in[0,t]$. 
Therefore, 
\bea\label{202508110350}
\sup_{t\in[0,t]}\max_{i\in[\sfd_\ep]}\big|z^{\ep,i}_s(\theta_\ep)-x^{\ep,i}_0\big|
&\leq&
\consti(1+\constt)\constj\> t(\max_{i\in[\sfd_\ep]}\sfk^\ep_i)
\exp\bigg(\consti\constj t(\max_{i\in[\sfd_\epsilon]}\sfk^\ep_i)\bigg).
\eea


We have
\bea\label{202508110356} &&
\bigg|t\sum_{j\in\checkN_i^\ep}(\theta^{ij}-\theta^{ij*}) \psi^{ij}\big(x^{\ep,i}_0,x^{\ep,j}_0\big)\bigg|^2
-
\bigg|z^{\ep,i}_t(\theta_\ep)-z^{\ep,i}_t(\theta_\ep^*)\bigg|^2
\nn\\&\leq&
\bigg|z^{\ep,i}_t(\theta_\ep)-z^{\ep,i}_t(\theta_\ep^*)
-t\sum_{j\in\checkN_i^\ep}(\theta^{ij}-\theta^{ij*}) \psi^{ij}\big(x^{\ep,i}_0,x^{\ep,j}_0\big)\bigg|
\nn\\&&\times
\bigg|(z^{\ep,i}_t(\theta_\ep)-x^{\ep,i}_0)-(z^{\ep,i}_t(\theta_\ep^*)-x^{\ep,i}_0)
+t\sum_{j\in\checkN_i^\ep}\theta^{ij} \psi^{ij}\big(x^{\ep,i}_0,x^{\ep,j}_0\big)
-t\sum_{j\in\checkN_i^\ep}\theta^{ij*} \psi^{ij}\big(x^{\ep,i}_0,x^{\ep,j}_0\big)\bigg|
\nn\\&\leq^{(\ref{202508110330}), (\ref{202508110350})}&
2
{\consti^2}(1+\constt){\constj^2}\>
(\max_{i\in[\sfd_\ep]}\sfk^\ep_i)^2t^2
\exp\bigg(\consti\constj t(\max_{i\in[\sfd_\epsilon]}\sfk^\ep_i)\bigg)\nn\\&&\times
3\consti(1+\constt)\constj t(\max_{i\in[\sfd_\ep]}\sfk^\ep_i)\max_{i\in[\sfd_\ep]}
\exp\bigg(\consti\constj t(\max_{i\in[\sfd_\epsilon]}\sfk^\ep_i)\bigg)
\nn\\&\leq&
6{\consti^3}(1+\constt)^2{(\constj)^3}
(\max_{i}\sfk^\ep_i)^3t^3
\exp\bigg(2\consti\constj t(\max_{i\in[\sfd_\epsilon]}\sfk^\ep_i)\bigg)
\nn\\&=:&
\eta(\ep,t)
\eea
for all $i\in[\sfd_\ep]$. 
By (\ref{202508110356}), we obtain
\beas&&
T^{-1}\int_0^{\ep^{\constd/3}}(\sfd_\ep)^{-1}\sum_{i\in[\sfd_\ep]}\big|z^{\ep,i}_t(\theta_\ep)-z^{\ep,i}_t(\theta_\ep^*)\big|^2dt
\nn\\&\geq&
T^{-1}\int_0^{\ep^{\constd/3}}(\sfd_\ep)^{-1}\sum_{i\in[\sfd_\ep]}
\bigg\{\bigg|t\sum_{j\in\checkN_i^\ep}(\theta^{ij}-\theta^{ij*}) \psi^{ij}\big(x^{\ep,i}_0,x^{\ep,j}_0\big)\bigg|^2
-\eta(\ep,t)\bigg\}
dt
\nn\\&\geq&
G^\ep(\theta_\ep,\constd)-I^\ep(\constd)
\eeas
for $\ep\leq T^{3/\constd}\wedge1$. 
This estimate is combined with (\ref{202508071017}) to show (\ref{202508100930}) since 
$
I^\ep(\constd)
\geq
T^{-1}\int_0^{\ep^{\constd/3}}\eta(\ep,t)dt
$.

\subsection{Proof of Lemma \ref{lem:delta}}
\label{sec:proof-delta}

$\Delta^\ep_r\leq\Delta^\ep_\infty$ by definition. Therefore we focus on 
$\Delta^\ep_\infty$.
First note that    
\beas &&
\max_{i\in[\sfd_\ep]}\sup_{s\in[0,t]}\big|x^{\ep,i}_s-z^{\ep,i}_s\big|
\nn\\&\leq&
\max_{i\in[\sfd_\ep]}\sup_{s\in[0,t]}\bigg|\sum_{j\in\checkN_i^\ep}
\int_0^{s}\theta^{ij*} \psi^{ij}\big(x^{\ep,i}_v,x^{\ep,j}_v\big)dv
-\sum_{j\in\checkN_i^\ep}
\int_0^{s}\theta^{ij*} \psi^{ij}\big(z^{\ep,i}_v,z^{\ep,j}_v\big)dv\bigg|
\nn\\&&
+\max_{i\in[\sfd_\ep]}\sup_{s\in[0,t]}\ep^\alpha\bigg|\int_0^s\sum_{j\in\checkN_i^\ep}V^{ij}(t,x^{\ep,i}_v,x^j_v)dv\bigg|
\nn\\&\leq&
\consti\constj(\max_{i}\sfk_\ep)\int_0^t
\max_{i\in[\sfd_\ep]}\sup_{s\in[0,v]}\big|x^{\ep,i}_s-z^{\ep,i}_s\big|dv
+\consti\ep^\alpha(\max_{i}\sfk_\ep)
\int_0^t\big(1+\max_{i\in[\sfd_\ep]}\big|x^{\ep,i}\big|\big)dv,
\eeas
and hence 
\bea\label{202508162212}
\Delta^\ep_\infty
 &=&
\max_{i\in[\sfd_\ep]}\sup_{s\in[0,T]}\big|x^{\ep,i}_s-z^{\ep,i}_s\big|
\nn\\&\leq^{(\ref{202508110733})}&
\consti\ep^\alpha(\max_{i}\sfk_\ep)T\bigg[1+
\bigg\{\max_{i\in[\sfd_\ep]}|x^{\ep,i}_0|
+
T\consti (\max_{i\in[\sfd_\ep]}\sfk^\ep_i)\bigg\}
\exp\big(T\consti (\max_{i\in[\sfd_\ep]}\sfk^\ep_i)\big)\bigg]
\nn\\&&\times
\exp\big(\consti\constj(\max_{i}\sfk^\ep_i)T\big)
\nn\\&\leq&
\consti\ep^\alpha \sfk_\ep T
\bigg\{1+\max_{i\in[\sfd_\ep]}|x^{\ep,i}_0|
+
 T\consti \sfk_\ep\bigg\}
\exp\big(T(\consti+\consti\constj)\sfk_\ep\big). 
\nn\\&&
\eea
As for the inequality (\ref{202508110733}) used above, routinely, 
\beas
\max_{i\in[\sfd_\ep]}\sup_{s\in[0,t]}
\big|x^{\ep,i}_s(\theta_\ep)\big|
&\leq&
\max_{i\in[\sfd_\ep]}|x^{\ep,i}_0|
+\max_{i\in[\sfd_\ep]}\sup_{s\in[0,t]}
\bigg|
\int_0^sS^{\ep,i}(v,\bbx^{\ep,i},\theta_\ep)dv\bigg|
\nn\\&\leq&
\max_{i\in[\sfd_\ep]}|x^{\ep,i}_0|
+\max_{i\in[\sfd_\ep]}
\int_0^t\consti \sum_{j\in\caln^\ep_i}(1+|x^{\ep,j}_v|)dv
\nn\\&\leq&
\max_{i\in[\sfd_\ep]}|x^{\ep,i}_0|
+
\int_0^t\consti (\max_{i\in[\sfd_\ep]}\sfk^\ep_i)\max_{i\in[\sfd_\ep]}(1+|x^{\ep,i}_v|)dv
\nn\\&\leq&
\max_{i\in[\sfd_\ep]}|x^{\ep,i}_0|
+
t\consti (\max_{i\in[\sfd_\ep]}\sfk^\ep_i)
+\consti (\max_{i\in[\sfd_\ep]}\sfk^\ep_i)
\int_0^t\max_{i\in[\sfd_\ep]}\sup_{s\in[0,v]}|x^{\ep,i}_s|dv
\eeas
for all $\theta_\ep\in\Theta_\ep$, 
and hence 
\bea\label{202508110733}
\sup_{\theta_\ep\in\Theta_\ep}
\max_{i\in[\sfd_\ep]}\sup_{s\in[0,t]}
\big|x^{\ep,i}_s(\theta_\ep)\big|
&\leq&
\bigg\{\max_{i\in[\sfd_\ep]}|x^{\ep,i}_0|
+
t\consti (\max_{i\in[\sfd_\ep]}\sfk^\ep_i)\bigg\}
\exp\big(t\consti (\max_{i\in[\sfd_\ep]}\sfk^\ep_i)\big).
\eea
%

\subsection{Proof of Lemma \ref{lem:omega-bound}}
\label{sec:proof-omega}

Let 
\beas 
q^\ep_i(\theta_\ep) 
&=& 
\big|\sum_{j\in\checkN_i^\ep}(\theta^{ij}_\ep-\theta^{ij*}_\ep) \psi^{ij}\big(x^{\ep,i}_0,x^{\ep,j}_0\big)\big|^2
\eeas
and 
$\wt{q}^\ep_i(\theta_\ep) =q^\ep_i(\theta_\ep) -E[q^\ep_i(\theta_\ep) ]$. 
Then 
\beas
\wt{q}^\ep_i(\theta_\ep) 
&=& 
\sum_{j_1,j_2\in\checkN_i^\ep}(\theta^{ij_1}_\ep-\theta^{ij_1*}_\ep)(\theta^{ij_2}_\ep-\theta^{ij_2*}_\ep)
\wt{\psi}^{\otimes2}_i(j_1,j_2),
\eeas
where 
\beas 
\wt{\psi}^{\otimes2}_i(j_1,j_2)
&=&
\psi^{ij_1}\big(x^{\ep,i}_0,x^{\ep,j_1}_0\big)\psi^{ij_2}\big(x^{\ep,i}_0,x^{\ep,j_2}_0\big)
-E\big[\psi^{ij_1}\big(x^{\ep,i}_0,x^{\ep,j_1}_0\big)\psi^{ij_2}\big(x^{\ep,i}_0,x^{\ep,j_2}_0\big)\big].
\eeas
Obviously, 
\beas
\esssup_\omega\max_{i\in[\sfd_\epsilon]}\max_{j_1,j_2\in N^\ep_i}\big|\wt{\psi}^{\otimes2}_i(j_1,j_2)\big|
&\leq&
2(\consti)^2(1+\constt)^2
\>=:\>
\constu.
\eeas

We need an intermediate result. 
\begin{lemma}\label{202508202025}
Under Conditions \ref{as:drift-lin},
\ref{as:mixing},
\bea\label{202508210502}
E\bigg[
\sup_{\theta_\ep\in\Theta_\ep}\bigg|
\sum_{i\in[\sfd_\ep]}\wt{q}^\ep_i(\theta_\ep)\bigg|\bigg]
&\leq&
{2^{-1/2}}(\constj)^2\constu\>\sfp_\ep(\max_{i\in[\sfp_\ep]}\sfk^\ep_i)^2
\bigg\{\sum_{h=0}^{n-1}(n-h)\alpha^\ep(h)\bigg\}^{1/2}. 
\eea
The left-hand side of (\ref{202508210502}) depends on $n$ through $\sfd_\ep=\sfp_\ep n$. 
\end{lemma}
\proof 

We have 
\bea\label{202508210427}
E\bigg[
\sup_{\theta_\ep\in\Theta_\ep}\bigg|
\sum_{i\in[\sfd_\ep]}\wt{q}^\ep_i(\theta_\ep)\bigg|\bigg]
&\leq&
4(\constj)^2
\sum_{r,r_1,r_2\in[\sfp_\ep]}
1_{\{r_1,r_2\in N^\ep_r\}}
E_{r,r_1,r_2}(i),
\eea
where 
\beas 
E_{r,r_1,r_2}(i)
&=&
E\bigg[
\bigg|
\sum_{i\in[\sfd_\ep]:\atop r(i)=r}
\wt{\psi}^{\otimes2}_i(r_1+u_\ep(i),r_2+u_\ep(i))
\bigg|
\bigg]. 
\eeas
In fact, 
\beas
\sum_{i\in[\sfd_\ep]}\wt{q}^\ep_i(\theta_\ep)
&=&
\sum_{i\in[\sfd_\ep]}
\sum_{j_1,j_2\in\checkN_i^\ep}
\bigg\{
(\theta^{ij_1}_\ep-\theta^{ij_1*}_\ep)(\theta^{ij_2}_\ep-\theta^{ij_2*}_\ep)
\wt{\psi}^{\otimes2}_i(j_1,j_2)
\bigg\}
\nn\\&=&
\sum_{r,r_1,r_2\in[\sfp_\ep]}
\sum_{i\in[\sfd_\ep]:\atop r(i)=r}
\sum_{j_1,j_2\in\checkN_i^\ep:\atop r(j_1)=r_1,r(j_2)=r_2}
\bigg\{
(\theta^{ij_1}_\ep-\theta^{ij_1*}_\ep)(\theta^{ij_2}_\ep-\theta^{ij_2*}_\ep)
\wt{\psi}^{\otimes2}_i(j_1,j_2)
\bigg\}
\nn\\&=&
\sum_{r,r_1,r_2\in[\sfp_\ep]}
(\theta^{rr_1}_\ep-\theta^{rr_1*}_\ep)(\theta^{rr_2}_\ep-\theta^{rr_2*}_\ep)
\sum_{i\in[\sfd_\ep]:\atop r(i)=r}\
\sum_{j_1,j_2\in\checkN_i^\ep:\atop r(j_1)=r_1,r(j_2)=r_2}
\wt{\psi}^{\otimes2}_i(j_1,j_2), 
\eeas
so that
\beas&&
\sup_{\theta_\ep\in\Theta_\ep}\bigg|
\sum_{i\in[\sfd_\ep]}\wt{q}^\ep_i(\theta_\ep)\bigg|
\nn\\&\leq&
4(\constj)^2
\sum_{r,r_1,r_2\in[\sfp_\ep]}
\bigg|
\sum_{i\in[\sfd_\ep]:\atop r(i)=r}
\sum_{j_1,j_2\in\checkN_i^\ep:\atop r(j_1)=r_1,r(j_2)=r_2}
\wt{\psi}^{\otimes2}_i(j_1,j_2)
\bigg|
\nn\\&=&
4(\constj)^2
\sum_{r,r_1,r_2\in[\sfp_\ep]}
\bigg|
\sum_{i\in[\sfd_\ep]:\atop r(i)=r}
\sum_{j_1,j_2\in[\sfd_\ep]}
1_{\{j_1,j_2\in\checkN_i^\ep,\>r(j_1)=r_1,r(j_2)=r_2\}}
\wt{\psi}^{\otimes2}_i(j_1,j_2)
\bigg|
\nn\\&=&
4(\constj)^2
\sum_{r,r_1,r_2\in[\sfp_\ep]}
\bigg|
\sum_{i\in[\sfd_\ep]:\atop r(i)=r}
1_{\{r_1,r_2\in N^\ep_r\}}
\wt{\psi}^{\otimes2}_i(r_1+u_\ep(i),r_2+u_\ep(i))
\bigg|
\nn\\&&
(\because\ j_1=r(j_1)+u_\ep(i)\text{ and }j_2=r(j_2)+u_\ep(i)\text{ when }j_1,j_2\in N^\ep_i\text{ due to (\ref{202508210319})})
\nn\\&=&
4(\constj)^2
\sum_{r,r_1,r_2\in[\sfp_\ep]}
1_{\{r_1,r_2\in N^\ep_r\}}
\bigg|
\sum_{i\in[\sfd_\ep]:\atop r(i)=r}
\wt{\psi}^{\otimes2}_i(r_1+u_\ep(i),r_2+u_\ep(i))
\bigg|. 
\eeas

\begin{comment}
\hrulefill
\beas
\caln_i^\ep&=&\caln_{r(i)}^\ep+u_\ep(i)\quad\text{with}\quad u_\ep(i)=\sfp_\ep\lfloor(i-1)/\sfp_\ep\rfloor. 
\eeas
Therefore, if $j_1\in\caln_i^\ep$, then $j_1=r(j_1)+u_\ep(i)$. \\
\hrulefill
\end{comment}

By the covariance inequality for $\alpha$-mixing (e.g. \cite{dedecker2003new}), 
\beas 
\big(E_{r,r_1,r_2}(i)\big)^2
&\leq&
E\bigg[
\bigg|
\sum_{i\in[\sfd_\ep]:\atop r(i)=r}
\wt{\psi}^{\otimes2}_i(r_1+u_\ep(i),r_2+u_\ep(i))
\bigg|^2
\bigg]
\nn\\&\leq&
\sum_{i\in[\sfd_\ep]:\atop r(i)=r}\sum_{i'\in[\sfd_\ep]:\atop r(i')=r}
\bigg|\text{Cov}\bigg[\wt{\psi}^{\otimes2}_i(r_1+u_\ep(i),r_2+u_\ep(i)),\>\wt{\psi}^{\otimes2}_i(r_1+u_\ep(i'),r_2+u_\ep(i'))\bigg]\bigg|
\nn\\&\leq&
4\constu^2\sum_{m\in[n]}\sum_{m'\in[n]}\alpha^\ep(|m-m'|)
\nn\\&\leq&
8\constu^2\sum_{h=0}^{n-1}(n-h)\alpha^\ep(h)
\eeas
for all $r,r_1,r_2\in[\sfp_\ep]$. 
Combine this with (\ref{202508210427}) to obtain 
\beas
E\bigg[
\sup_{\theta_\ep\in\Theta_\ep}\bigg|
\sum_{i\in[\sfd_\ep]}\wt{q}^\ep_i(\theta_\ep)\bigg|\bigg]
&\leq&
8\sqrt{2}(\constj)^2\constu
\sum_{r,r_1,r_2\in[\sfp_\ep]}
1_{\{r_1,r_2\in N^\ep_r\}}
\bigg\{\sum_{h=0}^{n-1}(n-h)\alpha^\ep(h)\bigg\}^{1/2}
\nn\\&\leq&
8\sqrt{2}(\constj)^2\constu\>\sfp_\ep(\max_{i\in[\sfp_\ep]}\sfk^\ep_i)^2
\bigg\{\sum_{h=0}^{n-1}(n-h)\alpha^\ep(h)\bigg\}^{1/2}. 
\eeas
This completes the proof. 
\qed

We stress the fact that, in the repeated measurements model, 
$Q^\ep(\theta_\ep)$ of (\ref{202508171722}), 
$\varphi_\ep$ of (\ref{202508171719}) and 
$\Omega^\ep_0(\nu)$ of (\ref{202508210601}) 
depend on $n$. 
From (\ref{202508171722}) and (\ref{202508171719}), 
\beas 
\varphi_\ep
&=& 
3^{-1}T^{-1}(\sfd_\ep)^{-1}
\sup_{\theta_\ep\in\Theta_\ep}\bigg|\sum_{i\in[\sfd_\ep]}\wt{q}^\ep_i(\theta_\ep)\bigg|. 
\eeas

Moreover, recall that 
from \eqref{202508210610}
we have
\bea \label{202508210610-2}
E\big[G^\ep(\theta_\ep,\constd)\big]
&\geq&
3^{-1}\constk T^{-1}\ep^\constv\rho_\ep(\theta_\ep,\theta_\ep^*)^2. 
\eea

Then, for $\Omega^\ep_0(\nu)$ of (\ref{202508210601}), 
\beas 
P\big[(\Omega^\ep_0(\nu))^c\big]
&=&
P\big[\varphi_\ep>2^{-1}\inf_{\theta_\ep\in U_\ep(\nu)^c}E[Q^\ep(\theta_\ep)]\big]
\nn\\&\leq&
P\bigg[
(\sfd_\ep)^{-1}
\sup_{\theta_\ep\in\Theta_\ep}\bigg|\sum_{i\in[\sfd_\ep]}\wt{q}^\ep_i(\theta_\ep)\bigg|
>2^{-1}\constk \ep^\beta \nu^2
\bigg]
\nn\\&&
(\because G^\ep(\theta_\ep,\constd)=\ep^\constd Q^\ep(\theta_\ep)\text{and }(\ref{202508210610-2}))
\nn\\&\leq&
(\sfd_\ep)^{-1}2\constk^{-1}
 \ep^{-\beta}
 \nu^{-2}
E\bigg[
\sup_{\theta_\ep\in\Theta_\ep}\bigg|\sum_{i\in[\sfd_\ep]}\wt{q}^\ep_i(\theta_\ep)\bigg|
\bigg]
\nn\\&\leq&
{16\sqrt{2}}(\constj)^2\constu\constk^{-1}
 \ep^{-\beta} 
\nu^{-2}
(\max_{i\in[\sfp_\ep]}\sfk^\ep_i)^2\>
n^{-1}
\bigg\{\sum_{h=0}^{n-1}(n-h)\alpha^\ep(h)\bigg\}^{1/2}
\eeas
by Lemma \ref{202508202025}.

\subsection{Proof of \autoref{thm:rates}}\label{sec:proof-rates}

We prove the following: for a given deterministic rate $r_\ep \to 0$, for every $\bar \nu > 0$,
\[
P\Big( \rho_\ep(\wh{\theta}_\ep^r,\theta_\ep^*) > r_\ep \bar \nu \Big) \to 0
\qquad\text{as }\ep\to0.
\]

Result  \textit{(i)} immediately follows from \autoref{thm:cons}. To see this, simply let $r_\epsilon^{-1} = \log \ep^{-1} $ and 
evaluate terms in the proof of \autoref{thm:cons} at  $\nu = \nu_\ep = \bar \nu  / \log \ep^{-1}$, for some positive, fixed $\bar \nu$.  The additional logarithmic term does not affect the convergence arguments, as every term in the proof converges at least polynomially in $\ep$. 

For the proof of $(ii)$, we set $r_\ep = \ep^\gamma$. Therefore, we need to evaluate the terms appearing in the proof of \autoref{thm:cons} at $\nu = \nu_\ep = \bar \nu \ep^{\gamma}$, for some positive, fixed $\bar \nu$. We describe below the key changes in the convergence arguments.

\paragraph{Convergence of $P\big[T(\ep,A_\ep)<T\big]$.}
The change in the rate for $n_\ep$, appearing in  $\sfd_\ep = n_\ep \sfp_\ep$, does not affect the convergence 
in \eqref{eq:te-bound}. 

\paragraph{Convergence of $\mathds 1_{\big\{
\constr\ep^{(\constv)/2}\nu_\ep<(I^\ep(\constd)^{1/2})\vee(4\Delta^\ep_r)\big\}}$.}
Equation \eqref{eq:step3-1}
\[
    \constr^2\ep^{\constv}\nu_\ep^2 \geq  I^\ep(\constd)
\]
is now determined by the behavior of $\ep^{4\constd/3 } / \ep^{2 \gamma + \constv}$;
the inequality is eventually satisfied as $\ep \to 0$ if
$c_1  > 3 (\beta + 2 \gamma)$, as in \eqref{eq:bc5}.
Similarly, due to \eqref{eq:delta-inf-bound}, the 
inequality 
\[ \constr\ep^{(\constv)/2}\bar \nu \ep^\gamma \geq4\Delta^\ep_\infty \]
holds eventually  if
$\alpha - \lS > (\constv)/2 + \gamma$, that is if \eqref{eq:bc6} holds.

\paragraph{Convergence of $\sfd_\ep e^{-(\xi^\ep_r)^2}$.}
In this case, under Assumption \ref{as:balance2}, 
\begin{align*}
    \xi^\ep_r & = 
\constr\constn A_\ep^{-1}\ep^{-1+(\constv)/2} \bar \nu  \ep^\gamma \exp\bigg(-\consto(\max_{i\in[\sfp_\ep]}\sfk^\ep_i)^{r-1}{\sf o}_\ep\bigg)
\\
& \gtrsim \left( \log \frac 1 \ep \right)^{-\frac{1 + \lA}{2(r-1)}} \ep^{-1+(\constv)/2 + \lS + \gamma } \to \infty. 
\end{align*}

\paragraph{Convergence of $P\big[(\Omega^\ep_0(\nu_\ep))^c\big]$.}

Again, as a consequence of Lemma \ref{lem:omega-bound}, under Assumption \ref{as:mixing} and \ref{as:balance2},
we have that
\begin{align}
    \label{eq:p-omega-bound}P\big[(\Omega^\ep_0(\nu))^c\big]
    &\leq
    {16\sqrt{2}}(\constj)^2\constu\constk^{-1}\>
    \bigg\{\sum_{h=0}^{n_\ep-1}(1-n_\ep^{-1}h)\alpha^\ep(h)\bigg\}^{1/2}
    \epsilon^{-\beta} \nu_\ep^{-2}
    \left (\max_{i\in[\sfp_\ep]}\sfk^\ep_i\right)^2\>
    n_\ep^{-1/2}
    \\ & \lesssim \frac{\sfk_\ep^2 \nu_\ep^{-2}}{\epsilon^{\beta}\sqrt{n_\ep}} \notag 
    \\ & \lesssim
    \left(\log\ep^{-1}\right)^{1/(r-1)} \frac{\ep^{2\gamma + \lN + \beta}}{\ep^{\beta + 2\gamma}}  \to 0. \notag 
\end{align}

In conclusion, by combining the above inequalities, we get 
\begin{align}\label{eq:err-bound-O}
    P\left( \rho_\epsilon(\wh{\theta}_\ep^r, \theta_\ep^*) > r_\epsilon \bar \nu \right)
    &\leq 
    \sfd_\ep e^{-(\xi^\ep_r)^2} + 
P\big[(\Omega^\ep_0(\nu_\ep))^c\big] + 
P\big[T(\ep,A_\ep)<T\big] 
\notag \\ & \quad 
 + \mathds 1_{\big\{
\constr\ep^{(\constv)/2}\nu_\ep<(I^\ep(\constd)^{1/2})\vee(4\Delta^\ep_r)\big\}}
\notag 
\\ & \lesssim
O\left(\left(\log\ep^{-1}\right)^{1/(r-1)}  \epsilon^{\lN} \right) 
\end{align}

\printbibliography

\end{document}